\documentclass[10pt]{article}
\usepackage{latexsym, amsmath, amssymb, amsthm}
\usepackage{enumerate, mathrsfs, bbm}
\usepackage{graphicx, tikz}
\usepackage{authblk, setspace, cite}

\usepackage[top=1 in, bottom=1 in, left=1.1 in, right=1.1 in]{geometry}
\usepackage[pdftex,linkcolor=red,citecolor=blue,backref=page]{hyperref}
\numberwithin{equation}{section}

\newcommand{\keywords}[1]{\par\noindent\textbf{Keywords:} #1}
\newtheorem{thm}{Theorem}[section]
\newtheorem{lem}[thm]{Lemma}
\newtheorem{defi}[thm]{Definition}

\def\be{\begin{equation}}     \def\ee{\end{equation}}

\title{Initial and initial-boundary value problems for a cubic sixth-order Boussinesq equation}
\author{Yuhao Xie\\
	School of Mathematical Sciences,\\
	University of Electronic Science and Technology of China,\\
	Chengdu, Sichuan 610054, China \\
	\texttt{xieyuhao@std.uestc.edu.cn}	\and
	Shenghao Li\footnote{Corresponding author} \\
	School of Mathematical Sciences,\\
	University of Electronic Science and Technology of China,\\
	Chengdu, Sichuan 610054, China \\
	\texttt{lish@uestc.edu.cn}}
\date{}

\begin{document}
	\maketitle
	\begin{abstract}
		In this paper,	we consider the  sixth-order Boussinesq equation with a cubic nonlinearity,
		\[	u_{tt}-u_{xx}+k u_{xxxx}-u_{xxxxxx}+\left( u^3 \right) _{xx}=0, \ \ \mbox{with }k=\pm 1,\]
		posed  on both $\mathbb{R}^+$ and $\mathbb{R}$. The local wellposedness are improved in related Bourgain-type spaces, $X^{s,b}$, for $s>-\frac12$ with $b=\frac12$ and $b>\frac12$ accordingly (See. \cite{esfahani2025well,zhong2024local}). The key ingredient is to refine the trilinear estimate on Burgain type space by adapting ideas from Tao's multiplier method, interpolation theory and our early work on IBVPs of dispersive equations \cite{li2020low,tao2001multilinear}.
	\end{abstract}
	\noindent \keywords{Sixth-order Boussinesq equation, Initial  value problem, Initial boundary value problem, Bourgain space.}
	\section{Introduction}\label{section 1}
	\hspace{1em}
	The original sixth order Boussinesq equation (SOBE),
	\be\label{1.3}
	u_{tt}-u_{xx}-u_{xxxx}-u_{xxxxxx}+\left( u^2 \right) _{xx}=0,
	\ee
	was first derived from the Euler equation under shallow waves assumption with a higer order correct by Christov, Maugin, and Velarde \cite{christov1996well} to remedy the ill-posedness of the well-known Boussinesq equation \cite{boussinesq1872theorie}:
	\be\label{1.4}
	u_{tt}-u_{xx}-u_{xxxx}+\left( u^2 \right) _{xx}=0.
	\ee
	The SOBE had been applied in ocean engineering, nonlinear optics, and nonlinear lattice dynamics in elastic crystals. However, in the derivation of the SOBE, certain nonlinear terms were neglected for simplicity of the model. Based on this motivation, Zhong and Li \cite{zhong2024local} derived the fully generalized SOBE(gSOBE),
	\be\label{1.5}
	u_{tt}-u_{xx}-u_{xxxx}-u_{xxxxxx}-\left( u^2 \right) _{xx}-\left( u^2 \right) _{xxxx}-\left( uu_{xx} \right) _{xx}-\left( u^3 \right) _{xx}=0,
	\ee
	that retrieved all missing nonlinear terms from  $abcd$ systems introduced by Bona, Chen, Saut \cite{bona2002boussinesq,bona2004boussinesq}.\par
	The studies for initial value problem (IVP) of SOBE:
	\be\label{1.6}
	\left\{ \begin{array}{l}
		u_{tt}-u_{xx}+ku_{xxxx}-u_{xxxxxx}=\left( f\left( u \right) \right) _{xx},\quad x, t\in \mathbb{R},\\
		u\left( x,0 \right) =\varphi \left( x \right) \in H^s\left( \mathbb{R} \right) ,\quad u_t\left( x,0 \right) =\psi ^{\prime\prime}\left( x \right) \in H^{s-1}\left( \mathbb{R} \right) ,\\
	\end{array} \right. 
	\ee
	are well-developed in the past decades due to the comprehensive theories on IVPs of dispersive equations including the KdV equation and the ``good'' Boussinesq equation. When the nonlinear term is of quadratic nonlinearity, $f\left( u \right) =u^2$, Esfahani and Farah \cite{esfahani2012local} showed that \eqref{1.6} is locally well-posed for $s>-\frac{1}{2}$ by using the suitable Bourgain-type space, $X^{s,b}$ with $b>\frac{1}{2}$. Later, Esfahani and Wang \cite{esfahani2014bilinear} improved their result to $s>-\frac{3}{4}$ applying the $\left[ k;Z \right]$-multiplier norm method introduced by Tao \cite{tao2001multilinear}. In addition, its related initial boundary value problem (IBVP) was also studied by Li, Chen, Yang and Zhang \cite{li2018non,li2020low,li2022lower} with related spaces accordingly. When \eqref{1.6} is concerned with the power-law-type nonlinearity, $f\left( u \right) =\left| u \right|^pu$ ($p>0$), Esfahani, Farah, and Wang \cite{esfahani2012global} proved that \eqref{1.6} is locally well-posed for,
	$$\left\{ \begin{array}{l}
	s\ge 0,\quad 0<p <4,\\
	s\ge 1,\quad p \ge 4,\\
	\end{array} \right.$$ which is mainly relying on the estimates of oscillatory integrals introduced by Kenig \cite{kenig1991oscillatory}. In addition, the high-dimensional case of \eqref{1.6}:
	\be\label{1.7}
	\left\{ \begin{array}{l}
		u_{tt}=\varDelta \left( u-k\varDelta u+\varDelta ^2u+f\left( u \right) \right) ,\ x\in \mathbb{R}^n,\ t\in\mathbb{R}\\
		u\left( x,0 \right) =\varphi \left( x \right) ,\ u_t\left( x,0 \right) =\psi \left( x \right) ,\\
	\end{array} \right.
	\ee was studied by Esfahani and Tesfahun \cite{esfahani2025well} with cubic nonlinearity, $f\left( u \right) =\pm u^3$. They established the well-posedness of \eqref{1.7} in $H^s\left( \mathbb{R}^n \right) \times H^{s-3}\left( \mathbb{R}^n \right)$ for,
	$$\left\{ \begin{array}{l}
	s>-\frac{1}{6},\quad \ \ n=1,\\
	s>\frac{2n-3}{6},\quad 2\le n\le 4,\\
	s>\frac{2n-5}{4},\quad n\ge 5,\\
	\end{array} \right.$$ 
	based on trilinear estimates in $X^{s,b}$ with appropriate localized Strichartz estimate. On the other hand, Zhong and Li \cite{zhong2024local} studied the IVP of gSOBE:
	\be\label{1.8}
	\left\{ \begin{array}{l}
		u_{tt}-u_{xx}+ku_{xxxx}-u_{xxxxxx}=\left( u^2 \right) _{xx}+\left( u^2 \right) _{xxxx}+\left( uu_{xx} \right) _{xx}+\left( u^3 \right) _{xx},\quad \ x,t\in \mathbb{R},\\
		u\left( x,0 \right) =\varphi \left( x \right) \in H^s\left( \mathbb{R} \right) ,\quad u_t\left( x,0 \right) =\psi ^{\prime\prime}\left( x \right) \in H^{s-1}\left( \mathbb{R} \right) ,\\
	\end{array} \right.
	\ee
	and showed that \eqref{1.8} is locally well-posed for $s>\frac{1}{2}$ under $X^{s,b}$ framework.\par 
	In this paper, we continue to work on the gSOBE for its IBVP as well as IVP and intend to improve the result in \cite{zhong2024local,esfahani2025well}. The theories for dispersive equations with quadratic nonlinearity is well-developed(see \cite{esfahani2012local,esfahani2014bilinear,li2020low,li2022lower,cavalcante2017initial,holmer2006initial,kenig1993cauchy,kenig1996bilinear,kishimoto2013sharp,li2023lower,li2022non,yang2024effect}), therefore, we will only focus on equation with cubic term, $\left(u^3\right)_{xx}$. 
	We study the IBVP of the cubic SOBE in the quarter plane, 
	\be\label{1.1}
	\left\{ \begin{array}{l}
		u_{tt}-u_{xx}+k u_{xxxx}-u_{xxxxxx}=\left( u^3 \right) _{xx},\quad x>0,\quad t>0,\\
		u\left( x,0 \right) =\varphi \left( x \right) ,\ u_t\left( x,0 \right) =\psi ^{\prime\prime}\left( x \right) ,\\
		u\left( 0,t \right) =h_1\left( t \right) ,\ u_{xx}\left( 0,t \right) =h_2\left( t \right) ,\ u_{xxxx}\left( 0,t \right) =h_3\left( t \right) ,\\
	\end{array} \right.
	\ee 
	as well as its initial-value problem(IVP),
	\be\label{1.2}
	\left\{ \begin{array}{l}
		u_{tt}-u_{xx}+k u_{xxxx}-u_{xxxxxx}=\left(  u^3 \right)_{xx},\quad x, t\in \mathbb{R},\\
		u\left( x,0 \right) =\varphi \left( x \right) ,\ u_t\left( x,0 \right) =\psi ^{\prime\prime}\left( x \right) ,\\
	\end{array} \right.
	\ee
	where $k=\pm1$. 
	Before presenting the main theorem,
	we first introduce some corresponding spaces for the cubic SOBE.
	\begin{defi}
		(\textbf{Spaces for IVP}) For $s,b\in \mathbb{R}$, $X^{s,b}$ denotes the completion of the Schwartz class $S\left( \mathbb{R}^2 \right)$ with,
		$$\lVert u \rVert _{X^{s,b}}:=\lVert \left< \xi \right> ^s\left< \left| \tau \right|-\phi \left( \xi \right) \right> ^b\hat{u}\left( \xi ,\tau \right) \rVert _{L_{\xi ,\tau}^{2}\left( \mathbb{R}^2 \right)},$$
		where $\phi \left( \xi \right) =\sqrt{\xi ^6+k\xi ^4+\xi ^2}$, $\left< x \right> =1+\left| x \right|$ and $\hat{u}$ denotes the Fourier transform on both time and space of $u$. We denote $X_{T}^{s,b}:=X^{s,b}|_{\mathbb{R}\times \left( 0,T \right)}$ with the quotient norm:
		$$\lVert u \rVert _{X_{T}^{s,b}}:=\underset{w\in X^{s,b}}{inf}\left\{ \lVert w \rVert _{X^{s,b}}:w\left( x,t \right) =u\left( x,t \right) \ on\ \mathbb{R}\times \left( 0,T \right) \right\}.$$
	\end{defi}
	In order to embed $C_tH_{x}^{s}$ into $X^{s,b}$ space, the IVP requires $b>\frac{1}{2}$(see \cite{esfahani2012local,esfahani2014bilinear,zhong2024local}), but for IBVP, it will force $b\le\frac{1}{2}$ due to certain restrictions(see \cite{colliander2002generalized,cavalcante2017initial,erdougan2016regularity,holmer2006initial,li2020low,li2022lower,li2023lower}). Similar to the studies of the KdV equation \cite{colliander2002generalized,holmer2006initial} and the Coupled-KdV system \cite{li2022non}, we introduce the modified Fourier restriction space $X^{s,b,\sigma}$ and an intermediate space $Z^{s,\sigma-1}$.
	\begin{defi}
		(\textbf{Spaces for IBVP}) For $s\in \mathbb{R}$, $\sigma>\frac{1}{2}$, $\varLambda _{\sigma}$, $X^{s,\frac{1}{2},\sigma}$ and $Z^{s,\sigma-1}$ denote the completion of the Schwartz class $S\left( \mathbb{R}^2 \right)$ with,
		$$\lVert u \rVert _{\varLambda _{\sigma}\left( \mathbb{R}^2 \right)}:=\lVert \chi _{\left| \xi \right|\le 1}\left< \tau \right> ^{\sigma}\hat{u}\left( \xi ,\tau \right) \rVert _{L_{\xi ,\tau}^{2}\left( \mathbb{R}^2 \right)},\quad\lVert u \rVert _{X^{s,\frac{1}{2},\sigma}\left( \mathbb{R}^2 \right)}:=\lVert u \rVert _{X^{s,\frac{1}{2}}\left( \mathbb{R}^2 \right)}+\lVert u \rVert _{\varLambda _{\sigma}\left( \mathbb{R}^2 \right)},$$
		$$\lVert u \rVert _{Z^{s,\sigma -1}}:=\lVert \left< \tau \right> ^{\frac{s}{3}+\frac{1}{2}-\sigma}\left< \left| \tau \right|-\phi \left( \xi \right) \right> ^{\sigma -1}\hat{u}\left( \xi ,\tau \right) \rVert _{L_{\xi ,\tau}^{2}\left( \mathbb{R}^2 \right)}.$$In addition, we define:
		$$X_{T}^{s,\frac{1}{2},\sigma}:=X^{s,\frac{1}{2},\sigma}\left( \mathbb{R}^+\times \left( 0,T \right) \right) ,\quad \mathcal{X}_{T}^{s,\frac{1}{2},\sigma}:=C\left( \left[ 0,T \right] ;H^s\left( \mathbb{R}^+ \right) \right) \cap X_{T}^{s,\frac{1}{2},\sigma}$$
		with the norm:
		$$\lVert u \rVert _{\mathcal{X}_{T}^{s,\frac{1}{2},\sigma}}:=\left( \underset{t\in \left( 0,T \right)}{sup}\lVert u\left( \cdot ,t \right) \rVert _{H^s\left( \mathbb{R}^+ \right)}^{2} +\lVert u \rVert _{X_{T}^{s,\frac{1}{2},\sigma}}^2 \right) ^{\frac{1}{2}}.$$
	\end{defi}
	Since theories on IVP of dispersive equations are relatively mature and studies on related IBVPs are more challenging. We will then focus on the IBVP of the cubic SOBE first and results are as follows.
	\begin{thm}\label{thm1.3}
		For any $s>-\frac{1}{2}$, there exists a $\sigma_0=\sigma_0(s)>\frac{1}{2}$ such that for any $\sigma\in \left(\frac{1}{2},\sigma_0 \right]$,
		\begin{equation}\label{1.9}
		\Big\lVert \left( \frac{\xi ^2\widehat{uvw}}{\phi \left( \xi \right)} \right) ^{\lor} \Big\rVert _{X^{s,\sigma-1}}\lesssim \lVert u \rVert _{X^{s,\frac{1}{2},\sigma}}\lVert v \rVert _{X^{s,\frac{1}{2},\sigma}}\lVert w \rVert _{X^{s,\frac{1}{2},\sigma}},
		\end{equation}
		\be\label{1.10}
		\Big\lVert \left( \frac{\xi ^2\widehat{uvw}}{\phi \left( \xi \right)} \right) ^{\lor} \Big\rVert _{Z^{s,\sigma -1}}\lesssim \lVert u \rVert _{X^{s,\frac{1}{2},\sigma}}\lVert v \rVert _{X^{s,\frac{1}{2},\sigma}}\lVert w \rVert _{X^{s,\frac{1}{2},\sigma}}.
		\ee
	\end{thm}
	With the above trilinear estimates, through a standard contraction mapping principal technique, the following local well-posedness can be obtained based on the studies of linearized problem of SOBE in \cite{esfahani2012local,esfahani2014bilinear,li2020low,li2022lower}.
	For IBVP, we denote $\vec{h}:=\left( h_1,h_2,h_3 \right) \in \mathcal{H}^s\left( \mathbb{R}^+ \right)$ with,
	$$\mathcal{H}^s\left( \mathbb{R}^+ \right) :=H^{\frac{s+1}{3}}\left( \mathbb{R}^+ \right) \times H^{\frac{s-1}{3}}\left( \mathbb{R}^+ \right) \times H^{\frac{s-3}{3}}\left( \mathbb{R}^+ \right).$$

	\begin{thm}\label{thm1.5}
		Let $s>-\frac{1}{2}$ be given, there exists $\sigma_0=\sigma_0\left( s \right) \in \left( \frac{1}{2},1 \right)$, then for any $\sigma\in\left( \frac{1}{2},\sigma _0 \right]$, we can find $T=T\left( s,\sigma \right) >0$ such that if, $$\left( \varphi ,\psi ,\vec{h} \right) \in H^s\left( \mathbb{R}^+ \right) \times H^{s-1}\left( \mathbb{R}^+ \right) \times \mathcal{H}^s\left( \mathbb{R}^+ \right) ,$$
		then the IBVP \eqref{1.1} admits a unique solution $u\in \mathcal{X}_{T}^{s,\frac{1}{2},\sigma}$, and the corresponding solution map is real analytic.
	\end{thm}
	In addition, we also present the conclusion on trilinear estimate related to the IVP of SOBE. 
	\begin{thm}\label{thm1.4}
		For any $s>-\frac{1}{2}$, there exists a $b_0=b_0(s)>\frac{1}{2}$ such that for any $b\in \left( \frac{1}{2},b_0 \right]$,
		\begin{equation}\label{1.11}
		\Big\lVert \left( \frac{\xi ^2\widehat{uvw}}{\phi \left( \xi \right)} \right) ^{\lor} \Big\rVert _{X^{s,b-1}}\lesssim \lVert u \rVert _{X^{s,b}}\lVert v \rVert _{X^{s,b}}\lVert w \rVert _{X^{s,b}}.
		\end{equation}
	\end{thm} 
	\noindent It will then lead to the local wellposedness below and one shall notice that we do improve such result when comparing to $s>\frac{1}{2}$ in \cite{zhong2024local} and $s>-\frac{1}{6}$ in \cite{esfahani2025well}.
	\begin{thm}\label{thm1.6}
		Let $s>-\frac{1}{2}$ be given, there exists $b_0=b_0\left( s \right) \in \left( \frac{1}{2},1 \right)$, then for any $b\in\left( \frac{1}{2},b_0 \right]$, we can find $T=T\left( s,b \right) >0$ such that if, $$\left( \varphi ,\psi \right) \in H^s\left( \mathbb{R} \right) \times H^{s-1}\left( \mathbb{R} \right),$$
		then the IVP \eqref{1.2} admits a unique solution $u\in X_{T}^{s,b}$, and the corresponding solution map is real analytic.
	\end{thm} 
	According to above statements, one can observe that the key to establish the trilinear estimates for $(u^3)_{xx}$ with various conditions of $b$ in $X^{s,b}$ space. There are mainly two approaches to address multilinear estimates for dispersive equations. \textbf{Method \uppercase\expandafter{\romannumeral1}} is known as $\left[ k;Z \right]$-multiplier norm method introduced by Tao \cite{tao2001multilinear}. One of the most important advantages of this method is that it can naturally lead to the optimal choice of the index $s$ in $X^{s,b}$  through delicate analysis. Under such a framework, there are two strategies to handle multilinear estimate:
	\begin{itemize}
		\item Direct analysis;
		\item reducing the higher order multilinear estimate into multiple related lower order ones.
	\end{itemize}In fact, the direct analysis idea was first conducted on KdV equation:
	\be\label{1.12}
	\left\{ \begin{array}{l}
		u_t+u_{xxx}+\left( \frac{u^2}{2} \right) _x=0,\ x,t\in \mathbb{R},\\
		u\left( x,0 \right) =\varphi\left( x \right) ,\\
	\end{array} \right.
	\ee
	by Kenig \cite{kenig1996bilinear} to obtain the local well-posedness for $s>-\frac{3}{4}$ on $X^{s,b}$ with $b>\frac{1}{2}$. This strategy was also used by Erdogan and Tzirakis to study the cubic
	nonlinear Schrödinger equation on the half line \cite{erdougan2016regularity},
	\be\label{1.13}
	\left\{ \begin{array}{l}
		iu_t+u_{xx}+\left| u \right|^2u=0,\ x>0,\ t>0,\\
		u\left( x,0 \right) =\varphi\left( x \right) ,\ u\left( 0,t \right) =h\left( t \right) ,\\
	\end{array} \right.
	\ee
	and the derivative cubic nonlinear
	Schrödinger equation on the half line \cite{erdougan2018derivative},
	\be\label{1.14}
	\left\{ \begin{array}{l}
		iu_t+u_{xx}-i\left( \left| u \right|^2u \right) _x=0,\,\,x>0,\,\,t>0,\\
		u\left( x,0 \right) =\varphi\left( x \right) ,\,\,u\left( 0,t \right) =h\left( t \right) ,\\
	\end{array} \right.
	\ee
	in $X^{s,b}$ for $b<\frac{1}{2}$ and $s>0$. We shall point out that the comprehensive multilinear estimate is applicable for dispersive equations with simple structure. For example, the Schrödinger equation, its $X^{s,b}$ norm is,
	$$\lVert u \rVert _{X^{s,b}}=\lVert \left< \xi \right> ^s\left< \tau +\xi ^2 \right> ^b\hat{u}\left( \xi ,\tau \right) \rVert _{L_{\xi ,\tau}^{2}},$$
	as well as the KdV equation, Coupled-KdV system and some other dispersive systems. But for the SOBE, its $X^{s,b}$ norm is defined as,
	$$\lVert u \rVert _{X^{s,b}}=\lVert \left< \xi \right> ^s\left< \left| \tau \right|-\left| \xi \right|\sqrt{\xi ^4+\xi ^2+1} \right> ^b\hat{u}\left( \xi ,\tau \right) \rVert _{L_{\xi ,\tau}^{2}}.$$
	Realling the discussion in \cite{esfahani2012local,esfahani2014bilinear,li2020low,li2022lower}, one needs to analysis more than ten cases for sign of both $\xi$ and $\tau$ in the resonance function  when conducting analysis  for related bilinear estimates. On the other hand, for the cubic SOBE, there will be hundreds of cases to be discussed, which seems too difficult to handle. The second strategy was introduced by Tao (see Lemma 3.7 in \cite{tao2001multilinear}). By such strategy, Tao proved the local well-posedness of the mKdV equation,
	\be\label{1.15}
	\left\{ \begin{array}{l}
		u_t+u_{xxx}+u^2u_x=0,\,\,x,t\in \mathbb{R},\\
		u\left( x,0 \right) =u_0\left( x \right) ,\\
	\end{array} \right. 
	\ee
	for $s\ge\frac{1}{4}$ \cite{tao2001multilinear} by reducing the trilinear estimate into several bilinear estimates and then, applied symmetry techniques to unify these bilinear estimates into a single bilinear estimate which makes the problem easier. This reduction strategy (refered as TT* decomposition) had been sucessfully applied for many dispersive equations and systems in the past two decades, such as modified Kawahara equations\cite{chen2009low} and some coupled systems \cite{carvajal2019sharp,zhang2021trilinear}. Inspired by this idea, Zhong and Li \cite{zhong2024local} established the local well-posedness of IVP \eqref{1.2} for $s>\frac{1}{2}$. However, the application of symmetry techniques may bring some loss in regularity and as a result, it does not always achieve the optimal conclusion. Indeed, Esfahani's recent result \cite{esfahani2025well}($s>-\frac{1}{6}$) does imply that $s>\frac{1}{2}$ is not sharp. Therefore, we intend to use the other approach for multilinear estimate to improve the previous result in \cite{zhong2024local,esfahani2025well}.\par 
	\textbf{\text{Method \uppercase\expandafter{\romannumeral2}}} is based on $L^pL^q$ estimates. Under such framework, there are also two strategies:
	\begin{itemize}
		\item Restricting the solution in the standard $L^pL^q$ spaces for $s>0$;
		\item handling the multilinear estimate in $X^{s,b}$ through the help of $L^pL^q$ estimates.
	\end{itemize}
	The first case was introduced by Kenig \cite{kenig1993well} to study the generalized KdV equation:
	\be\label{1.16}
	\left\{ \begin{array}{l}
		u_t+u_{xxx}+u^{p-1}u_x=0,\,\,x,t\in \mathbb{R},\\
		u\left( x,0 \right) =\varphi\left( x \right) ,\\
	\end{array} \right.
	\ee
	where $p\in\mathbb{Z}^+$ and $p\ge2$. The second case was first conducted by Kenig \cite{kenig1993cauchy} to establish the local well-posedness of KdV equation \eqref{1.12} for $s>-\frac{5}{8}$. This strategy was later generilized by Huo and Jia to study the IVPs of fifth-order shallow water equations \cite{jia2009well}:
	$$u_t+\alpha u_{xxxxx}+\beta u_{xxx}+\gamma u_x+\mu f_1(u) =0,\ x,t\in \mathbb{R},$$ 
	where $\alpha\ne0$, $\alpha$, $\beta$, $\gamma\in\mathbb{R}$, $\mu\in\mathbb{C}$, $f_1(u)=\left( u^2+u^3 \right) _x+uu_{xx}+u_xu_{xx}$ and fourth-order nonlinear Schrödinger equations \cite{huo2005cauchy,huo2007refined}:
	$$iu_t+u_{xx}+\nu u_{xxxx}=f_2\left( u \right) ,\ x,t\in \mathbb{R},$$
	where $f_2\left( u \right) =-\frac{1}{2}\left| u \right|^2u+\lambda _1\left| u \right|^4u+\lambda _2\left( u_x \right) ^2\bar{u}+\lambda _3\left| u_x \right|^2u+\lambda _4u^2\left( \bar{u} \right) _{xx}+\lambda _5\left| u \right|^2u_{xx}$, $\lambda_j$ are some parameters depending on $\mu$ and $\nu$, with $\nu\ne0$ $\mu$, $\nu\in\mathbb{R}$. Their method can be applied to other dispersive equations and systems but only for their IVPs. In this article, we will first adapt such idea to trilinear estimates for the IBVP along with techniques introduced by Li, Yang, and Zhang \cite{li2022lower,li2020low,li2022non,li2023lower,li2018non,li2019wellposedness}.
	Moreover, we will also refine such idea and apply it on IVP \eqref{1.2}. We do believe that our method is generalizable to other dispersive equations and systems.\par 
	Before giving an explanation of our approach, let us first recall Tao's view for multilinear estimate in \cite{tao2001multilinear}. In fact, for a dispersive edquation with phase function $h(\xi)$, to show local well-posedness for IVP, we need to establish multilinear estimate in the form: 
	\be\label{1.17}
	\lVert G\left( u \right) \rVert _{X_{\tau =h\left( \xi \right)}^{s,b-1}}\lesssim\lVert u \rVert _{X_{\tau =h\left( \xi \right)}^{s,b}}^{k},
	\ee
	for a $k$-linear function $G$ of $u$ and $b>\frac{1}{2}$. Then, by duality, the $k$-linear estimate is equivalent to bound, 
	\be\label{1.18}
	\frac{\int_{A_{k+1}\left( \mathbb{R} \right)}{\left| \frac{\left< \xi _{k+1} \right> ^s\prod_{j=1}^{k+1}{m_j\left( \xi _j \right) f_j\left( \xi _j,\tau _j \right)}}{\left< \tau _{k+1}-h\left( \xi _{k+1} \right) \right> ^{1-b}\prod_{j=1}^k{\left< \xi _j \right> ^s\left< \tau _j-h\left( \xi _j \right) \right> ^b}} \right|}}{\prod_{j=1}^{k+1}{\lVert f_j \rVert _{L_{\xi}^{2}L_{\tau}^{2}}}},
	\ee
	where $m_j$ are some monomials, $f_j\in L^2(\mathbb{R}^2)$, $$F_{b}^{(j)}:=\left( \frac{m_j\left( \xi _j \right) f_j\left( \xi _j,\tau _j \right)}{\left< \xi _j \right> ^s\left< \tau _j-h\left( \xi _j \right) \right> ^b} \right) ^{\lor}\in S\left( \mathbb{R}^2 \right) \quad j=1,2,...,k,$$ $$F_{1-b}^{(k+1)}:=\left( \frac{\left< \xi _{k+1} \right> ^sm_{k+1}\left( \xi _{k+1} \right) f_{k+1}\left( \xi _{k+1},\tau _{k+1} \right)}{\left< \tau _{k+1}-h\left( \xi _{k+1} \right) \right> ^{1-b}} \right) ^{\lor}\in S\left( \mathbb{R}^2 \right),$$ and $A_k$ denotes the ``hyperplane":
	\begin{equation*}
	A_k\left( \mathbb{R} \right) :=\left\{ \left( \vec{\xi}, \vec{\tau} \right) :\sum_{j=1}^k{\xi _j}=\sum_{j=1}^k{\tau _j}=0 \right\} ,
	\end{equation*}
	with $\vec{\xi }=\left( \xi _1,...,\xi _k \right)$, $\vec{\tau }=\left( \tau _1,...,\tau _k \right)$, $k\in\mathbb{N}$, $k\ge3$. Next, according to the following lemma:
	\begin{lem}\label{lem1.7}(\cite{huo2005cauchy,huo2007refined,jia2009well})
		If $f_j\in S\left( \mathbb{R}^2 \right)$ , $k\in\mathbb{Z}^+$, $k\ge3$, then,
		$$\int_{A_k}{\prod_{j=1}^k{\widehat{f_j}\left( \xi _j,\tau _j \right)}=\int_{\mathbb{R}^2}{\left( \prod_{j=1}^k{f_j} \right) \left( x,t \right) dxdt}}.$$
	\end{lem} 
	One has, \eqref{1.18} is same as:
	\be\label{1.19}
	\frac{\int_{\mathbb{R}^2}{\left( F_{1-b}^{(k+1)}\prod_{j=1}^k{F_{b}^{(j)}} \right) \left( x,t \right) dxdt}}{\prod_{j=1}^{k+1}{\lVert f_j \rVert _{L_{\xi}^{2}L_{\tau}^{2}}}}.
	\ee
	Therefore, \eqref{1.18} can be achieved by Holder's inequality if we have some appropriate $L_{x}^{p}L_{t}^{q}$ estimates (the so-called Strichartz estimate) of $F_{b}^{(j)}$. With the above process, we are able to improve the previous conclusions for IVP \eqref{1.2} to $s>-\frac{1}{2}$.\par 
	As we point out in previous note, for IBVPs, some of the linear estimates will force $b\le\frac{1}{2}$, which is very different from IVPs. To overcome this kind of issue, we observe that the $L^p_xL^q_t$ estimates for SOBE, $$\widehat{F_{\rho}}\left( \xi ,\tau \right) :=\frac{f\left( \xi ,\tau \right)}{\left< \left| \tau \right|-\phi \left( \xi \right) \right> ^{\rho}}$$ (where the modified phase function $\phi$ is defined later) usually requires $\rho>\frac{1}{2}$ but we find out that one can still lower the requirement of $\rho$ by interpolation of these $L^p_xL^q_t$ estimates. Moreover, Li,Yang and Zhang\cite{li2018non,li2019wellposedness,li2020low,li2022lower,li2023lower,li2022non} developed a refined method for bilinear estimate on different dispersive equations to fullfil for various values of $b$ in $X^{s,b}$ spaces. By delicate analysis, we can use such idea to satisfy different requirement on $\rho$. Consequently, we can show the local wellposdness of IBVP \eqref{1.1} for $s>-\frac{1}{2}$.\par
	We point out that this method can be also applied for bilinear estimates. For bilinear estimates, one of the advantages of this method is that it needs less tedious discussion than $\left[ k;Z \right]$-multiplier norm method, especially for SOBE. Such application of this method will be shown in our future work. However, it seems that this approach can not directly indicate the optimal conclusion of index $s$. According to Kenig's early work on KdV equation, reader can observe that there still exists a gap between $-\frac{5}{8}$ \cite{kenig1993cauchy} and $-\frac{3}{4}$ \cite{kenig1996bilinear}. The similar situation also happens to bilinear estimate of SOBE. The problem for the reason remains open.\par
	
	This paper is organized as follows. In section \ref{section2}, we prove some preliminary estimates which play an important role in the proof of trilinear estimates. In section \ref{section3}, multilinear estimates for the IBVP on Bourgain spaces are established. Finally, we refine theory on $X^{s,b}$ space for IVP and adapt it to improve the conclusion for SOBE.
	\section{Notations and preliminary estimates}\label{section2}
	We denote,
	$$\widehat{F_{\rho}}\left( \xi ,\tau \right) :=\frac{f\left( \xi ,\tau \right)}{\left< \left| \tau \right|-\phi\left( \xi \right) \right> ^{\rho}},\quad D_{x}^{s}f:=\left( \left| \xi \right|^s\hat{f}^x\left( \xi \right) \right) ^{\lor},\quad J_{x}^{s}f:=\left( \left< \xi \right> ^s\hat{f}^x\left( \xi \right) \right) ^{\lor},$$
	$$\lVert f \rVert _{L_{x}^{p}L_{t}^{q}}:=\left( \int_{\mathbb{R}}{\left( \int_{\mathbb{R}}{\left| f\left( x,t \right) \right|^q}dt \right) ^{\frac{p}{q}}dx} \right) ^{\frac{1}{p}},\quad \lVert f \rVert _{L_{x}^{p}L_{T}^{q}}:=\left( \int_{\mathbb{R}}{\left( \int_{-T}^T{\left| f\left( x,t \right) \right|^q}dt \right) ^{\frac{p}{q}}dx} \right) ^{\frac{1}{p}},$$
	where $\land$, $\lor$ denote the Fourier transform and the inverse Fourier transform. Moreover, we introduce the modified unitary groups for SOBE:
	$$U_1\left( t \right) g:=\int_{\mathbb{R}}{e^{i\left( x\xi +t\phi\left( \xi \right) \right)}\hat{g}\left( \xi \right) d\xi},
	\quad U_2\left( t \right) g:=\int_{\mathbb{R}}{e^{i\left( x\xi -t\phi\left( \xi \right) \right)}\hat{g}\left( \xi \right) d\xi},$$
	and Fourier restriction operators:
	$$P^nf:=\int_{\left| \xi \right|\ge n}{e^{ix\xi}\hat{f}\left( \xi \right) d\xi},\quad P_nf:=\int_{\left| \xi \right|\le n}{e^{ix\xi}\hat{f}\left( \xi \right) d\xi},$$
	which are used to decompose a function $f$ as $f=P^nf+P_nf$ by frequency.
	\subsection{Modified phase function}
	We introduce a function $\left| \xi \right|^3+\frac{k}{2}\left| \xi \right|$ which has the following equivalence relation in
	the related Bourgain type space (for the proof, reader can refer \cite{esfahani2012local}):
	$$\lVert u \rVert _{X^{s,b}\left( \mathbb{R}^2 \right)}\sim \lVert \left< \xi \right> ^s\left< \left| \tau \right|-\left| \xi \right|^3-\frac{k}{2}\left| \xi \right| \right> ^b\widehat{u}\left( \xi ,\tau \right) \rVert _{L_{\xi ,\tau}^{2}\left( \mathbb{R}^2 \right)}.$$ This modified phase function helps us to estimate the resonance function(see \textbf{Lemma} \ref{lem5.1}) and has close relation for the phase function of KdV equation. For simplicity, we only present the proof for $k=-1$ since the case when $k=1$ can be addressed similarly. In the rest of this acticle, we denote $\phi \left( \xi \right) :=\left| \xi \right|^3-\frac{1}{2}\left| \xi \right|$ and introduce, 
	$$\phi _1\left( \xi \right):=\left\{ \begin{array}{l}
	\phi(\xi) ,\quad \  \xi >0,\\
	-\phi(\xi) ,\ \ \xi \le 0.\\
	\end{array} \right.$$ 
	One can notice that if $\left| \xi \right|\ge \frac{1}{2}$, then $\phi _{1}^{\prime}=3\xi ^2-\frac{1}{2}\ge \xi ^2>0$, thus, if $\left| \xi \right|\ge \frac{1}{2}$, then function $\phi_1$ is invertible, which will be very useful for our later study. Moreover, it can be verified that $\phi \left( \xi \right),$ $-\phi \left( \xi \right)\in \mathcal{A}$, which is a function class introduced in \cite{kenig1991oscillatory} by Kenig. Namely, $\phi$ satisfies:
	\begin{itemize}
		\item $\phi \in C^3\left( \mathbb{R}\setminus \left\{ 0 \right\} \right)$,
		\item the set $S_{\phi }:=\left\{ \xi \in \mathbb{R}\cup \left\{ \pm \infty \right\} ;\phi ^{\prime\prime}\left( \xi \right) =0\ or\  \underset{\xi \rightarrow \widetilde{\xi }}{\lim}\phi ^{\prime\prime}\left( \widetilde{\xi } \right)  =\phi ^{\prime\prime}\left( \xi \right) =\pm \infty \right\}=\left\{0,\pm \infty \right\}$ is finite,
		\item there exists constants $C_1,C_2,\epsilon_1>0$ such that for $\left| \xi \right|\le \epsilon_1 $, $$C_1\left| \xi \right|\le \phi ^{\prime\prime}\left( \xi \right)=6\left| \xi \right|\le C_2\left| \xi \right|,$$
		\item there exists constants $C_3,C_4,\epsilon_2>0$ such that for $\left| \xi \right|>\frac{1}{\epsilon _2}  $, $$C_3\left| \xi \right|\le \phi^{\prime\prime}\left( \xi \right)=6\left| \xi \right|\le C_4\left| \xi \right|,$$
		\item $\phi^{\prime\prime}\left( \xi \right)$ has a finite number of changes of monotonicity since $\phi^{\prime\prime}\left( \xi \right)=6\left| \xi \right|,$
	\end{itemize}
	and so does $-\phi$. Thus, the following lemma in \cite{kenig1991oscillatory} can be applied.
	\begin{lem}\label{lem1.8}
		In the domain $\Omega\subseteq \mathbb{R}\setminus \left\{ 0 \right\}$, for $h=\phi \left( \xi \right)$ or $-\phi \left( \xi \right)$, then,
		$$\lVert I\left( x,t \right) \rVert _{L_{x}^{4}L_{t}^{\infty}}\le c_h\left( \int_{\varOmega}{\left| \hat{f}\left( \xi \right) \right|^2\left| \frac{h^{\prime}\left( \xi \right)}{h^{\prime\prime}\left( \xi \right)} \right|^{\frac{1}{2}}d\xi} \right) ^{\frac{1}{2}},$$
		where $I\left( x,t \right) :=\int_{\mathbb{R}}{e^{ith \left( \xi \right) +ix\xi}\widehat{f}\left( \xi \right) d\xi},$ and $c_h$ is a constant only depends on $h$.
	\end{lem}
	\subsection{Theorems for interpolation}
	We persent the interpolation theorems which are frequently used through the paper.
	\begin{lem}(\textbf{Riesz-Thorin Interpolation})\label{lem2.1}
		Let $\left( X,\mathcal{M},\mu \right)$, $\left( Y,\mathcal{N},\nu \right)$ be positive, $\sigma$-finite measure spaces. Assume $1\le p_j,q_j\le \infty$ and constants $M_j\ge0$, $j=1,2$. Suppose that we are given a linear mapping $T$,
		$$T:\left( L^{p_0}\left( \mu \right) +L^{p_1}\left( \mu \right) \right) \longrightarrow \left( L^{q_0}\left( \nu \right) +L^{q_1}\left( \nu \right) \right)$$
		with,
		$$\lVert Tf \rVert _{q_j}\le M_j\lVert f \rVert _{p_j},$$
		for any $f\in L^{p_j}\left( \mu \right)$. Then, for any $\theta\in\left[0,1\right]$,
		$$\lVert Tf \rVert _{q_\theta}\le M_\theta\lVert f \rVert _{p_\theta},$$
		where $\frac{1}{p_{\theta}}=\frac{1-\theta}{p_0}+\frac{\theta}{p_1}$, $\frac{1}{q_{\theta}}=\frac{1-\theta}{q_0}+\frac{\theta}{q_1}$, $M_{\theta}=M_{0}^{1-\theta}M_{1}^{\theta}.$
	\end{lem}
	For details of \textbf{Lemma} \ref{lem2.1}, readers can refer \cite{liriesz,lofstrom1976interpolation,stein1956interpolation}. Next, we will present the interpolation theorem for the mixed norm spaces $L^{\vec{p}}$ with the norm:
	$$\lVert f \rVert _{L^{\vec{p}}}:=\lVert f\left( x_1,...,x_n \right) \rVert _{L_{x_1}^{p_1}...L_{x_n}^{p_n}},$$
	where $\vec{p}=(p_1,...,p_n)$, $1\le p_j\le \infty$ and for the linear operator with index. This interpolation theorem comes from \cite{benedek1961space} and actually, is a generalization of Stein interpolation theorem.
	\begin{lem}\label{lem2.2}
		Suppose that $\left( X_j,M_j,\mu_j \right)$, $\left( Y_j,N_j,\nu_j \right)$ are totally $\sigma$-finite measure spaces, $\left( X,M,\mu \right) =\left( \prod_{j=1}^n{X_j},\prod_{j=1}^n{M_j},\prod_{j=1}^n{\mu _j} \right)$, $\left( Y,N,\nu \right) =\left( \prod_{j=1}^n{Y_j},\prod_{j=1}^n{N_j},\prod_{j=1}^n{\nu _j} \right)$.
		Let ${\left\{ T_{z} \right\}}$ be a linear operator family of admissible growth\footnote{For the detailed definition of admissible growth and the following $H(M)$, $w$, reads can refer Theorem 1 in \cite{benedek1961space}.}, $z$ is a complex number, $1\le \vec{p}_j,\vec{q}_j\le\infty$, $0<\theta<1$. Assume that for $f\in H\left(M\right)$ and
		$$\lVert T_{i\theta}f \rVert _{L^{\vec{q}_0}}\le A_0\left( y \right) \lVert f \rVert _{L^{\vec{p}_0}},\quad \lVert T_{1+i\theta}f \rVert _{L^{\vec{q}_1}}\le A_1\left( y \right) \lVert f \rVert _{L^{\vec{p}_1}}$$
		with
		$$\log \left| A_j\left( y \right) \right|\le Ae^{a\left| y \right|},\ a<\pi ,\ j=0,1.$$
		Then,
		$$\lVert T_t\left( f \right) \rVert _{L^{\vec{q}_{\theta}}}\le A_{\theta}\lVert f \rVert _{L^{\vec{p}_{\theta}}},$$
		where
		$$\frac{1}{\vec{p}_{\theta}}=\frac{1-\theta}{\vec{p}_0}+\frac{\theta}{\vec{p}_1},\quad \frac{1}{\vec{q}_{\theta}}=\frac{1-\theta}{\vec{p}_0}+\frac{\theta}{\vec{q}_1},$$
		$$\log \left| A_\theta \right|\le \int_{\mathbb{R}}{w\left( 1-\theta,y \right) \log A_0\left( y \right) dy}+\int_{\mathbb{R}}{w\left( \theta,y \right) \log A_1\left( y \right) dy}.$$
	\end{lem}
	\subsection{Estimates for $U_j(t)$}
	In the rest of the paper, we will fix $n>\frac{1}{2}$. 
	\begin{lem}\label{lem2.3}
		The group $\left\{ U_j\left( t \right) \right\} _{-\infty}^{+\infty}$, $j=1,2$, satisfies:
		\be\label{2.1}
		\lVert D_xP^{n}U_j\left( t \right) g \rVert _{L_{x}^{\infty}L_{t}^{2}} \lesssim \lVert g \rVert _{L^2},
		\ee
		\be\label{2.2}
		\lVert D_{x}^{-\frac{1}{4}}P^{n}U_j\left( t \right) g \rVert _{L_{x}^{4}L_{t}^{\infty}} \lesssim \lVert g \rVert _{L^2},
		\ee
	\end{lem}
	\begin{proof}
		Firstly, as mentioned before, $\eta=\phi_1\left(\xi\right)$ is invertible when $\left| \xi \right|\ge n>\frac{1}{2}$, that is, $\xi =\phi _{1}^{-1}\left(\eta\right):=\eta^{-1}$, then,
		\begin{equation*}
		\begin{aligned}
		P^{n}U_1\left( t \right) g&=\int_n^{+\infty}{e^{i\left( x\xi +t\eta \right)}\hat{g}\left( \xi \right) d\xi}+\int_{-\infty}^{-n}{e^{i\left( x\xi -t\eta \right)}\hat{g}\left( \xi \right) d\xi}\\
		&=\int_{\phi_1(n)}^{+\infty}{e^{i\left( x\eta^{-1}+t\eta \right)}\hat{g}\left( \eta^{-1} \right) \frac{d\eta}{\eta^{\prime}}}+\int_{-\infty}^{\phi_2(-n)}e^{i\left( x\eta^{-1}-t\eta \right)}\hat{g}\left( \eta^{-1} \right) \frac{d\eta}{\eta^{\prime}}\\
		&:=\check{h}_{1}^{\eta}+\hat{h}_{2}^{\eta}
		\end{aligned}
		\end{equation*}
		where $h_1=\chi_{\eta\ge \phi_1(n)}e^{ix\eta^{-1}}\hat{g}\left( \eta^{-1} \right) \frac{1}{\eta^{\prime}}$, $h_2=\chi_{\eta\le \phi_1(-n)}e^{ix\eta^{-1}}\hat{g}\left( \eta^{-1} \right) \frac{1}{\eta^{\prime}}$. Moreover, one has,
		\begin{equation*}
		\begin{aligned}
		\lVert h_1 \rVert _{L_{\eta}^{2}}^{2}&=\int_{\mathbb{R}}{\left| \chi_{\eta\ge \phi_1(n)}e^{ix\eta^{-1}}\hat{g}\left( \eta^{-1} \right) \frac{1}{\eta^{\prime}} \right|}^2d\eta _2\\
		&\lesssim \int_{\phi_1(n)}^{+\infty}\left| \hat{g}\left( \eta^{-1} \right) \frac{1}{\eta^{\prime}} \right|^2d\eta\\
		&=\int_{n}^{+\infty}{\left| \hat{g}\left( \xi \right) \frac{1}{\eta^{\prime}\left( \xi \right)} \right|^2\eta^{\prime}\left( \xi \right) d\xi}\lesssim \int_n^{+\infty}{\frac{\left| \hat{g}\left( \xi \right) \right|^2}{\left| \xi \right|^2}d\xi}.
		\end{aligned}
		\end{equation*}
		Similarly,
		$$\lVert h_2 \rVert _{L_{\eta}^{2}}^{2}\lesssim\int_{-\infty}^{-n}{\frac{\left| \hat{g}\left( \xi \right) \right|^2}{\left| \xi \right|^2}d\xi},$$
		therefore, we can conclude that,
		$$\lVert D_xP^{n}U_1\left( t \right) g \rVert _{L_{t}^{2}}^{2}\lesssim\lVert g \rVert _{L^2}.$$
		The proof for $j=2$ can be obtained similarly, thus we omit. Next, we turn to proof \eqref{2.2}. By direct computation, we observe that $\left| \frac{\phi ^{\prime}\left( \xi \right)}{\phi ^{\prime\prime}\left( \xi \right)} \right|=\left| \frac{3\xi ^2-\frac{1}{2}}{6\xi} \right|\le \frac{1}{2}\left| \xi \right|$. Moreover, according to \textbf{Lemma} \ref{lem1.8} with $\phi \left( \xi \right),$ $-\phi \left( \xi \right)\in \mathcal{A}$, for $n>\frac{1}{2}$, one has,\begin{equation*}
		\begin{aligned}
		\lVert P^nU_j\left( t \right) g \rVert _{L_{x}^{4}L_{t}^{\infty}}&=\lVert U_j\left( t \right) P^ng \rVert _{L_{x}^{4}L_{t}^{\infty}}\\
		&\lesssim\left( \int_{\left| \xi \right|\ge n}{\left| \widehat{P^ng} \right|^2\left| \frac{\phi ^{\prime}\left( \xi \right)}{\phi ^{\prime\prime}\left( \xi \right)} \right|^{\frac{1}{2}}d\xi} \right) ^{\frac{1}{2}}\\
		&\lesssim\left( \int_{\left| \xi \right|\ge n}{\left| \widehat{P^ng} \right|^2\left| \xi \right|^{\frac{1}{2}}d\xi} \right) ^{\frac{1}{2}}=\left( \int_{\left| \xi \right|\ge n}{\left| \widehat{g}\left( \xi \right) \right|^2\left| \xi \right|^{\frac{1}{2}}d\xi} \right) ^{\frac{1}{2}}\lesssim\lVert P^n g \rVert _{H^{\frac{1}{4}}},
		\end{aligned}
		\end{equation*}
		which implies \eqref{2.2}.
	\end{proof}
	The following lemma gives the decay estimate for SOBE, similar conclusion had been proved in Lemma 4.3 of \cite{esfahani2011stability}, but the phase function here is slightly different. 
	\begin{lem}\label{lem2.4} Let $$G\left( x,t \right) :=\int_{\mathbb{R}}{e^{it\phi\left( \xi \right) +ix\xi}}d\xi,$$
		then for any $\delta >0$, there exists a constant $C>0$ such that for any $0<\left| t \right|\le\delta$,
		\begin{equation}\label{2.3}
		\underset{x\in \mathbb{R}}{sup}\left| G\left( x,t \right) \right|\le C\left| t \right|^{-\frac{1}{3}},
		\end{equation}
		where $C$ only depends on $\delta$ and $\phi$.
	\end{lem}
	\begin{proof}
		The proof is standard, we omit for simplicity, readers can refer \cite{cui2005strichartz}.
	\end{proof}
	By the decay estimate, we can obtain the following lemma.
	\begin{lem}\label{lem2.5}
		For $j=1,2$,
		\be\label{2.4}
		\lVert U_j\left( t \right) g \rVert _{L_{x}^{p}}\lesssim \left| t \right|^{-\frac{1}{3}\left( 1-\frac{2}{p} \right)}\lVert g \rVert _{L^{p^{\prime}}},
		\ee
		where $\frac{1}{p}+\frac{1}{p^{\prime}}=1$, $p\ge2$.
		\begin{proof}
			First we observe that $U_1\left( t \right) g\left( x \right) =\left( G\left( \cdot ,t \right) *g \right) \left( x \right) $, then by young's inequality and \textbf{Lemma} \ref{lem2.2},
			\begin{equation}\label{2.5}
			\lVert U_1\left( t \right) g \rVert _{L_{x}^{\infty}}\le \lVert G\left( x,t \right) \rVert _{L_{x}^{\infty}}\lVert g \rVert _{L^1}\lesssim\left| t \right|^{-\frac{1}{3}}\lVert g \rVert _{L^1}.
			\end{equation} 
			Moreover, it can be checked that, $$\lVert U_1\left( t \right) g \rVert _{L_{x}^{2}}=\left\lVert \left( e^{it\phi\left( \xi \right)}\widehat{g}\left( \xi \right) \right) ^{\lor}\left( x,t \right) \right\rVert _{L_{x}^{2}}=\lVert e^{it\phi\left( \xi \right)}\widehat{g}\left( \xi \right) \rVert _{L_{\xi}^{2}}=\lVert g \rVert _{L^2},$$
			thus, one has,
			\be\label{2.6}
			\lVert U_1\left( t \right) g \rVert _{L_{x}^{2}}\lesssim \lVert g \rVert _{L^2}.
			\ee
			Let $Tg:=U_1\left( t \right) g$, we conduct an interpolation by \textbf{Lemma} \ref{2.1} between \eqref{2.5} and \eqref{2.6}, it leads to,
			$$\lVert T_1\left( t \right) g \rVert _{L_{x}^{p_{\theta}}}\lesssim\left| t \right|^{-\frac{1-\theta}{3}}\lVert g \rVert _{L^{q_{\theta}}},$$
			where $\frac{1}{p_{\theta}}=\frac{1-\theta}{\infty}+\frac{\theta}{2}$, $\frac{1}{q_{\theta}}=\frac{1-\theta}{1}+\frac{\theta}{2},$ which is equivalent to $\theta =\frac{2}{p_{\theta}}$ and $\frac{1}{p_{\theta}}+\frac{1}{q_{\theta}}=1$. This proves \eqref{2.4} for $j=1$. For $j=2$, the result follows by the fact that $U_2\left( t \right) g=U_1\left( -t \right) g$.
		\end{proof}
	\end{lem}
	Next, we turn to prove the Strichartz estimates($L^pL^q$ estimates) for SOBE.
	\begin{lem}\label{lem2.6}
		\be\label{2.7}
		\lVert U_j\left( t \right) g \rVert _{L_{t}^{q}L_{x}^{p}}\lesssim\lVert g \rVert _{L^2}, \quad j=1,2,
		\ee
		with $\frac{2}{q}=\frac{1}{3}\left( 1-\frac{2}{p} \right)$, $2\le p,q\le8$.
	\end{lem}
	\begin{proof}
		We shall use the TT* argument(interested reader can refer Lemma 2.2 in \cite{ginibre1995generalized}) and Hardy-littlewood-Sobolev inequality which claims that,
		\be\label{2.8}
		\lVert f*\left| \cdot \right|^{m} \rVert _{L^a}\lesssim\lVert f \rVert _{L^b},
		\ee
		where $1<b<a<\infty$, $0<m <1$ and $1+\frac{1}{a}=m +\frac{1}{b}$.\\
		For $j=1$, let $U_{1}^{*}\left( t \right)$ be the duality of $U_1(t)$, then we can obtain,
		\begin{equation*}
		\begin{aligned}
		\left< U_{1}^{*}f,g \right> _{L^2}=\left< f,U_1g \right> _{L^2}=&\int_{\mathbb{R}}{\int_{\mathbb{R}}{f\left( t,x \right) \overline{\int_{\mathbb{R}}{e^{it\phi\left( \xi \right) +ix\xi}\widehat{g}\left( \xi \right) d\xi}}dxdt}}\\
		=&\int_{\mathbb{R}}{\int_{\mathbb{R}}{\overset{\land}{f^x}\left( t,\xi \right) e^{-it\phi \left( \xi \right)}\overline{\widehat{g}}\left( \xi \right)}d\xi dt}\\
		=&\int_{\mathbb{R}}{\int_{\mathbb{R}}{\left( \overset{\land}{f^x}\left( t,\xi \right) e^{-it\phi \left( \xi \right)} \right) ^{\lor}\overline{g}\left( x \right)}dxdt}
		=&\int_{\mathbb{R}}{\int_{\mathbb{R}}{U_2\left( t \right) f\left( t,\cdot \right) dt\overline{g}\left( x \right)}dx}
		\end{aligned}
		\end{equation*}
		by the fact that $U_2\left( t \right) =U_1\left( -t \right)$. Thus, $$U_{1}^{*}f\left( t,\cdot \right) =\int_{\mathbb{R}}{U_2\left( t \right) f\left( t,\cdot \right) dt},\quad U_1U_{1}^{*}f\left( t,\cdot \right) =\int_{\mathbb{R}}{U_1\left( \tau -t \right) f\left( t,\cdot \right) dt}.$$ 
		Hence, by TT* argument, 
		$$\lVert U_1g \rVert _{L_{t}^{q}L_{x}^{p}}\lesssim \lVert g \rVert _{L^2}$$
		is equivalent to:
		$$\lVert U_1U_{1}^{*}f \rVert _{L_{t}^{q}L_{x}^{p}}\lesssim \lVert f \rVert _{L_{t}^{q^\prime}L_{x}^{p^\prime}},$$
		where $\frac{1}{p}+\frac{1}{p^{\prime}}=1$, $\frac{1}{q}+\frac{1}{q^{\prime}}=1$.
		Then, by \eqref{2.4}, one has,\begin{equation*}
		\begin{aligned}
		\lVert U_1U_{1}^{*}f \rVert _{L_{x}^{p}}=&\int_{\mathbb{R}}{\lVert U_1\left( \tau -t \right) f\left( t,\cdot \right) \rVert _{L_{\xi}^{p}}}dt\\
		\lesssim&\int_{\mathbb{R}}{\left| \tau -t \right|^{-\frac{1}{3}\left( 1-\frac{2}{p} \right)}\lVert f \rVert _{L_{x}^{p^{\prime}}}}dt=\lVert f \rVert _{L_{x}^{p^{\prime}}}*\left| t \right|^{-\frac{1}{3}\left( 1-\frac{2}{p} \right)},
		\end{aligned}
		\end{equation*}
		therefore, by \eqref{2.8}, for $\frac{1}{3}\left( 1-\frac{2}{p} \right) +q^{\prime}=1+\frac{1}{q}$, $q^{\prime}=\frac{q}{q-1}$, we have,
		$$\lVert U_1U_{1}^{*}f \rVert _{L_{t}^{q}L_{x}^{p}}\lesssim \lVert f \rVert _{L_{t}^{q^\prime}L_{x}^{p^\prime}}.$$ For $j=2$, \eqref{2.7} follows by the fact that $U_2\left( t \right) =U_1\left( -t \right)$.
	\end{proof}
	Finally, the following lemma gives the estimate on $L_{x}^{2}L_{T}^{\infty}$ norm for $U_j(t)$.
	\begin{lem}\label{lem2.7}
		For any $\delta>0$ and $s>\frac{3}{4}$ there exists a constant $C$ depending on $\delta$, $s$ such that for any $g\in H^s\left( \mathbb{R} \right)$, $0<T\le \delta$,
		\be\label{2.9}
		\lVert U_j\left( t \right) g \rVert _{L_{x}^{2}L_{T}^{\infty}}\le C\lVert g \rVert _{H^s}.
		\ee
	\end{lem}
	\begin{proof}
		The proof is standard, interested reader can refer Theorem 2.8 in \cite{cui2005strichartz}.
	\end{proof}
	\subsection{Estimates for $F_\rho$}
	We intend to provide estimates for $F_{\rho}$ on various norm of $L^pL^q$ which are the main tool for later theorem.
	\begin{lem}\label{lem2.8}
		If $\rho >\frac{1}{2}$, then for fixed $N_0>0$,
		\begin{equation}\label{2.10}
		\lVert P_{N_0}F_{\rho} \rVert _{L_{x}^{2}L_{t}^{\infty}}\lesssim\lVert f \rVert _{L_{\xi}^{2}L_{\tau}^{2}}.
		\end{equation}
	\end{lem}
	\begin{proof}
		First we let $\widehat{H_{\rho}}\left( \xi ,\tau \right) :=\frac{f\left( \xi ,\tau \right)}{\left< \tau \right>^{\rho}}$, then for $\rho>\frac{1}{2}$,
		\begin{equation*}
		\begin{aligned}
		\lVert H_{\rho} \rVert _{L_{x}^{2}}=&\Big\lVert \left( \frac{f\left( \xi ,\tau \right)}{\left< \tau \right>^{\rho}} \right) ^{\lor} \Big\rVert _{L_{x}^{2}}\\
		=&\Big\lVert \int_{-\infty}^{+\infty}{\frac{f\left( \xi ,\tau \right) e^{it\tau}}{\left< \tau \right> ^{\rho}}d\tau} \Big\rVert _{L_{\xi}^{2}}\\
		\le&\Big\lVert \int_{-\infty}^{+\infty}{\left| \frac{f\left( \xi ,\tau \right)}{\left< \tau \right> ^{\rho}} \right|d\tau} \Big\rVert _{L_{\xi}^{2}}\le\Big\lVert \lVert f\left( \xi ,\tau \right) \rVert _{L_{\tau}^{2}}\lVert \left< \tau \right> ^{\rho} \rVert _{L_{\tau}^{2}} \Big\rVert _{L_{\xi}^{2}}\lesssim\lVert f \rVert _{L_{\xi}^{2}L_{\tau}^{2}},
		\end{aligned}
		\end{equation*}
		Note the fact that $\frac{\chi _{\left| \xi \right|\le n}\left( 1+\left| \tau \right| \right) ^{\rho}}{\left( 1+\left| \left| \tau \right|-\phi _1\left( \xi \right) \right| \right) ^{\rho}}$ is bounded, we have \eqref{2.10}.
	\end{proof}
	\begin{lem}\label{lem2.9}
		For $\rho >\frac{1}{2}$, $s>\frac{3}{4}$, one has, \be\label{2.11}
		\lVert D_xP^nF_{\rho} \rVert _{L_{x}^{\infty}L_{t}^{2}}\lesssim\lVert f \rVert _{L_{\xi}^{2}L_{\tau}^{2}},
		\ee
		\be\label{2.12}
		\lVert D_{x}^{-\frac{1}{4}}P^nF_{\rho} \rVert _{L_{x}^{4}L_{t}^{\infty}}\lesssim\lVert f \rVert _{L_{\xi}^{2}L_{\tau}^{2}}.
		\ee
		\be\label{2.13}
		\lVert J_{x}^{-s}F_{\rho} \rVert _{L_{x}^{2}L_{t}^{\infty}}\lesssim \lVert f \rVert _{L_{\xi}^{2}L_{\tau}^{2}}.
		\ee
	\end{lem}
	\begin{proof}
		Using the change of variable $\lambda _1=\tau -\phi \left( \xi \right)$, $\lambda _2=\tau +\phi \left( \xi \right)$, we write:\begin{equation*}
		\begin{aligned}
		F_{\rho}&=\int_{-\infty}^{+\infty}{\int_{-\infty}^{+\infty}{\frac{f\left( \xi ,\tau \right)}{\left< \left| \tau \right|-\phi \left( \xi \right) \right> ^{\rho}}e^{i\left( x\xi +t\tau \right)}d\xi d\tau}}\\
		&=\int_0^{+\infty}{\int_{-\infty}^{+\infty}{\frac{f\left( \xi ,\tau \right)}{\left< \tau -\phi\left( \xi \right) \right> ^{\rho}}e^{i\left( x\xi +t\tau \right)}d\xi d\tau}}+\int_{-\infty}^0{\int_{-\infty}^{+\infty}{\frac{f\left( \xi ,\tau \right)}{\left< \tau +\phi\left( \xi \right) \right> ^{\rho}}e^{i\left( x\xi +t\tau \right)}d\xi d\tau}}\\
		&=\int_{-\phi\left( \xi \right)}^{+\infty}{\int_{-\infty}^{+\infty}{\frac{f\left( \xi ,\lambda _1+\phi \left( \xi \right) \right)}{\left< \lambda _1 \right> ^{\rho}}e^{i\left( x\xi +t\phi \left( \xi \right) +t\lambda _1 \right)}d\xi d\lambda _1}}\\
		&+\int_{-\infty}^{\phi\left( \xi \right)}{\int_{-\infty}^{+\infty}{\frac{f\left( \xi ,\lambda _2-\phi \left( \xi \right) \right)}{\left< \lambda _2 \right> ^{\rho}}e^{i\left( x\xi -t\phi \left( \xi \right) +t\lambda _2 \right)}d\xi d\lambda _2}}\\
		&:=U_1\left( t \right) g_1+U_2\left( t \right) g_2,
		\end{aligned}
		\end{equation*}
		where $\widehat{g_1}=\int_{-\phi\left( \xi \right)}^{+\infty}{\frac{f\left( \xi ,\lambda _1+\phi\left( \xi \right) \right)}{\left< \lambda _1 \right> ^{\rho}}e^{it\lambda _1}d\lambda _1}$, $\widehat{g_2}=\int_{-\infty}^{\phi\left( \xi \right)}{\frac{f\left( \xi ,\lambda _2-\phi\left( \xi \right) \right)}{\left< \lambda _2 \right> ^{\rho}}e^{it\lambda _2}d\lambda _2}$. Then, for $\rho>\frac{1}{2}$, we can conclude that,\begin{equation*}
		\begin{aligned}
		&\lVert D_xP^nU_1\left( t \right) g_1 \rVert _{L_{x}^{\infty}L_{t}^{2}}\\
		\lesssim&\lVert g_1 \rVert _{L^2}\\
		\lesssim&\Bigg\lVert \int_{-\phi \left( \xi \right)}^{+\infty}{\left| \frac{f\left( \xi ,\lambda _1+\phi \left( \xi \right) \right)}{\left< \lambda _1 \right> ^{\rho}} \right|d\lambda _1} \Bigg\rVert _{L_{\xi}^{2}}\lesssim\lVert f\left( \xi ,\tau \right) \rVert _{L_{\xi}^{2}L_{\tau}^{2}}
		\end{aligned}
		\end{equation*}
		by \eqref{2.1}. Simiarly,
		$$\lVert D_xP^nU_2\left( t \right) g_2 \rVert _{L_{x}^{\infty}L_{t}^{2}}\lesssim\lVert f\left( \xi ,\tau \right) \rVert _{L_{\xi}^{2}L_{\tau}^{2}},$$
		which proves \eqref{2.11}, then, \eqref{2.12}, \eqref{2.13} can also be proved by the similar argument, thus we omit.
	\end{proof}
	The above conclusion requiers $\rho>\frac{1}{2}$, for later use, we will need to lower the requirement of $\rho$.
	\begin{lem}\label{lem2.10}
		For $\rho >\frac{2}{3}\left( 1-\frac{2}{p} \right)$, $2\le p\le8$, one has,
		\be\label{2.14}
		\lVert F_{\rho} \rVert _{L_{x}^{p}L_{t}^{p}}\lesssim\lVert f \rVert _{L_{\xi}^{2}L_{\tau}^{2}}.
		\ee
		Especially, for $\rho >\frac{2}{9}$, then,
		\be\label{2.15}
		\lVert F_{\rho} \rVert _{L_{x}^{3}L_{t}^{3}}\lesssim\lVert f \rVert _{L_{\xi}^{2}L_{\tau}^{2}}.
		\ee
		For $\rho >\frac{1}{3}$, then,
		\be\label{2.16}
		\lVert F_{\rho} \rVert _{L_{x}^{4}L_{t}^{4}}\lesssim\lVert f \rVert _{L_{\xi}^{2}L_{\tau}^{2}}.
		\ee
		For $\rho >\frac{4}{9}$, then,
		\be\label{2.17}
		\lVert F_{\rho} \rVert _{L_{x}^{6}L_{t}^{6}}\lesssim\lVert f \rVert _{L_{\xi}^{2}L_{\tau}^{2}}.
		\ee
	\end{lem}
	\begin{proof}
		By \textbf{Lemma} \ref{lem2.6}, take $p=q=8$ in \eqref{2.7}, one has,
		\be\label{2.18}
		\lVert U_j\left( t \right) g \rVert _{L_{x}^{8}L_{t}^{8}}\lesssim\lVert g \rVert _{L^2},
		\ee
		recalling the proof of \textbf{Lemma} \ref{lem2.7}, we can just repeat the argument and obtain that:
		\be\label{2.19}\lVert F_{\rho} \rVert _{L_{x}^{8}L_{t}^{8}}\lesssim\lVert U_1\left( t \right) g_1 \rVert _{L_{x}^{8}L_{t}^{8}}+\lVert U_2\left( t \right) g_2 \rVert _{L_{x}^{8}L_{t}^{8}}\lesssim\lVert f \rVert _{L_{\xi}^{2}L_{\tau}^{2}}
		\ee
		for $\rho>\frac{1}{2}$, then, let $Tf:=F_{\theta \rho}$, by interpolation\footnote{In the rest of the paper, it always means that we take advantage of \textbf{Lemma} \ref{lem2.2} when we use interpolation.} (\textbf{Lemma} \ref{lem2.2}) between \eqref{2.19} and the trivial estimate (Plancherel’s equality),
		\be\label{2.20}
		\lVert F_0 \rVert _{L_{x}^{2}L_{t}^{2}}\lesssim\lVert f \rVert _{L_{\xi}^{2}L_{\tau}^{2}},
		\ee
		we can obtain that, 
		$$\lVert F_{\theta \rho} \rVert _{L_{x}^{p_{\theta}}L_{t}^{q_{\theta}}}\lesssim\lVert f \rVert _{L_{\xi}^{2}L_{\tau}^{2}},$$
		where $\frac{1}{p_{\theta}}=\frac{1}{q_{\theta}}=\frac{1-\theta}{2}+\frac{\theta}{8},$ which is equivalent to $\theta =\frac{4}{3}\left( 1-\frac{2}{p_{\theta}} \right)$ and $\theta \rho =\frac{4}{3}\left( 1-\frac{2}{p_{\theta}} \right) \rho >\frac{2}{3}\left( 1-\frac{2}{p_{\theta}} \right) $. This gives \eqref{2.14}.
	\end{proof}
	In addition, we also need some other interpolations of these estimates. 
	\begin{lem}\label{lem2.11}
		For $\rho>\frac{1}{2}$, one has,
		\be\label{2.21}
		\lVert D_{x}^{\frac{1}{6}}P^nF_{\rho} \rVert _{L_{x}^{6}L_{t}^{6}}\lesssim\lVert f \rVert _{L_{\xi}^{2}L_{\tau}^{2}},
		\ee
		for $\rho>\frac{\theta}{2}$, $\theta\in[0,1]$, one has,
		\be\label{2.22}
		\lVert D_{x}^{\theta}P^nF_{\rho} \rVert _{L_{x}^{\frac{2}{1-\theta}}L_{t}^{2}}\lesssim\lVert f \rVert _{L_{\xi}^{2}L_{\tau}^{2}}.
		\ee
	\end{lem}
	\begin{proof}
		\eqref{2.21} follows by the interpolation of \eqref{2.11} and \eqref{2.12}, \eqref{2.22} is the straight rusult of the interpolation between \eqref{2.11} and \eqref{2.20}.
	\end{proof}
	
	\section{Studies on the IBVP}\label{section3}
	
	\subsection{Related linear estimates for IBVP}
	In this subsection, we will present some existing conclusions and related notations on IBVP. In the rest of the paper, we let $\eta \left( t \right) \in C_{0}^{\infty}\left( \mathbb{R} \right)$ be a cut-off function satisfying
	$$\eta \left( t \right) \equiv 1,\ t\in \left( -1,1 \right) ,\ supp\ \eta \subset \left( -2,2 \right) .$$
	First, we consider the linearized IVP as below:
	\be\label{3.1}
	\left\{ \begin{array}{l}
		u_{tt}-u_{xx}+\beta u_{xxxx}-u_{xxxxxx}=f_{xx}\left( x,t \right) ,\ x\in\mathbb{R},\quad t\ge0,\\
		u\left( x,0 \right) =\ \varphi \left( x \right) ,\ u_t\left( x,0 \right) =\ \psi ^{\prime\prime}\left(x\right),\\
	\end{array} \right.
	\ee
	a solution of \eqref{3.1} can be given by:
	$$u\left( x,t \right) =\left[ W_{\mathbb{R}}\left( \varphi ,\psi \right) \right]\left(x,t\right)+\int_0^t{\left[ W_{\mathbb{R}}\left( 0,f \right) \right] \left( x,t-\tau \right) d\tau},$$
	with
	$$\left[ W_{\mathbb{R}}\left( \varphi ,\psi \right) \right]\left(x,t\right):=\left[ V_1\left( \varphi \right) \right] \left( x,t \right) +\left[ V_2\left( \psi \right) \right] \left( x,t \right),$$
	$$\left[ V_1\left( \varphi \right) \right] \left( x,t \right) :=\frac{1}{2}\int_{\mathbb{R}}{\left( e^{it\phi \left( \xi \right) +ix\xi}+e^{-it\phi \left( \xi \right) +ix\xi} \right) \hat{\varphi}\left( \xi \right) d\xi},$$
	$$\left[ V_2\left( \psi \right) \right] \left( x,t \right) :=\frac{1}{2i}\int_{\mathbb{R}}{\left( e^{it\phi \left( \xi \right) +ix\xi}-e^{-it\phi \left( \xi \right) +ix\xi} \right) \frac{\xi ^2\hat{\psi}\left( \xi \right)}{\phi \left( \xi \right)}d\xi}.$$
	
	Next, we consider the related IBVP as below:
	\be\label{3.2}
	\left\{ \begin{array}{l}
		u_{tt}-u_{xx}+\beta u_{xxxx}-u_{xxxxxx}=f_{xx}\left(x,t \right),\quad x>0,\quad t>0\\
		u\left( x,0 \right) =\varphi \left( x \right) ,\ u_t\left( x,0 \right) =\psi ^{\prime\prime}\left( x \right) ,\\
		u\left( 0,t \right) =h_1\left( t \right) ,\ u_{xx}\left( 0,t \right) =h_2\left( t \right) ,\ u_{xxxx}\left( 0,t \right) =h_3\left( t \right) ,\\
	\end{array} \right.
	\ee
	a solution of \eqref{3.2} can be written as,
	$$u\left( x,t \right) =a\left( x,t \right) +c\left( x,t \right) +W_{bdr}\left( \vec{h}-\vec{a}-\vec{c} \right),$$
	where,
	\be\label{3.3}
	a\left( x,t \right) :=\left[ W_{\mathbb{R}}\left( \varphi ^*,\psi ^* \right) \right] \left( x,t \right),
	\ee
	\be\label{3.4}
	c\left( x,t \right) :=\int_0^t{\left[ W_{\mathbb{R}}\left( 0,f^* \right) \right] \left( x,t-\tau \right) d\tau},
	\ee
	$\varphi ^*$, $\psi ^*$ are the zero extension of $\varphi$, $\psi$: $\mathbb{R}^+\longrightarrow \mathbb{R}$, $f^{*}$ is an extension from $\mathbb{R}^{+}$ to $\mathbb{R}$ and,
	\be\label{3.5}
	\vec{a}\left( x,t \right) :=\left( a_1,a_2,a_3 \right) =\left( a,a_{xx},a_{xxxx} \right) \left( 0,t \right),
	\ee
	\be\label{3.6}
	\vec{c}\left( x,t \right) :=\left( c_1,c_2,c_3 \right) =\left( c,c_{xx},c_{xxxx} \right) \left( 0,t \right),
	\ee
	with $W_{bdr}$ being the boundary integral operator, for definition $W_{bdr}$ and the details on the derivation of the representation, readers can refer \cite{li2019wellposedness,li2020low,li2022lower}, we omit here for simplicity.\par 
	The following estimates \textbf{Lemma} \ref{lem 3.3}, \ref{lem 3.4} come from \cite{li2018non,li2019wellposedness,li2020low,li2022lower}.
	\begin{lem}\label{lem 3.3}
		Let $s\in\mathbb{R}$, $\frac{1}{2}<\sigma<1$, $0<T\le1$. Then there exists a constant $C=C(s,\sigma)$ such that,
		$$\|W_{\mathbb{R}}\left( \varphi ,\psi \right)\|_{H^s(\mathbb{R})}+\lVert  W_{\mathbb{R}}\left( \varphi ,\psi \right) \rVert _{X^{s,\frac{1}{2},\sigma}}\le C\left( \lVert \varphi \rVert _{H^s\left( \mathbb{R} \right)}+\lVert \psi \rVert _{H^{s-1}\left( \mathbb{R} \right)} \right),$$
		$$\Big\lVert \eta \left( \frac{t}{T}\right) \int_0^t{W_{\mathbb{R}}\left( 0,f \right) \left( x,t-\tau \right) d\tau}  \Big\rVert _{X^{s,\frac{1}{2},\sigma}}\le C \Big\lVert \left( \frac{\xi ^2\hat{f}}{\phi \left( \xi \right)} \right) ^{\lor} \Big\rVert  _{X^{s,\sigma -1}}.$$
	\end{lem}
	\begin{lem}\label{lem 3.4}
		For $s\in\mathbb{R}$, there exists a $C=C(s)$ such that for any $\sigma\in\left(\frac{1}{2},1\right)$, one has,
		$$\underset{x\ge 0}{sup}\lVert \partial _{x}^{j}\left[ \eta \left( t \right) W_{bdr}\left( \vec{h} \right) \right] \left( x,\cdot \right) \rVert _{{H}^{\frac{s-j+1}{3}}\left( \mathbb{R}^+ \right)}\le C\lVert \vec{h} \rVert _{{\mathcal{H}}^s\left( \mathbb{R}^+ \right)},$$
		$$\underset{t\ge 0}{sup}\lVert \left[ \eta \left( t \right) W_{bdr}\left( \vec{h} \right) \right] \left( \cdot ,t \right) \rVert _{H^s\left( \mathbb{R}^+ \right)}\le C\lVert \vec{h} \rVert _{{\mathcal{H}}^s\left( \mathbb{R}^+ \right)},$$
		$$\lVert \left[ \eta \left( t \right) W_{bdr}\left( \vec{h} \right) \right] \left( \cdot ,t \right) \rVert _{X^{s,\frac{1}{2},\sigma}}\le C\lVert \vec{h} \rVert _{{\mathcal{H}}^s\left( \mathbb{R}^+ \right)},$$
		where $j=1,2,...,5.$
	\end{lem}
	\begin{lem}\label{lem 3.5}
		For any $s\in\mathbb{R}$ and $\frac{1}{2}<\sigma\le1$, there exists a constant $C=C\left(s,\sigma\right)$ such that
		\begin{equation*}
		\underset{x\ge 0}{sup}\lVert \eta \left( t \right) \int_0^t{\partial _{x}^{j}\left[ W_{\mathbb{R}}\left( 0,f \right) \right] \left( x,t-\tau \right) d\tau} \rVert _{H_{t}^{\frac{s-j+1}{3}}\left( \mathbb{R} \right)}\lesssim\Big\lVert \left( \frac{\xi ^2\hat{f}}{\phi \left( \xi \right)} \right) ^{\lor} \Big\rVert _{X^{s,\sigma -1}}+\Big\lVert \left( \frac{\xi ^2\hat{f}}{\phi \left( \xi \right)} \right) ^{\lor} \Big\rVert _{Z^{s,\sigma -1}},
		\end{equation*}
		where $j=0,1,2,3,4,5$.
	\end{lem}
	\begin{proof}
		The proof is standard, similar results can be found in \cite{cavalcante2017initial,colliander2002generalized,li2022non,holmer2006initial}, thus we omit the proof for simplicity. But we point out that if $j=0,3$, then when $-1\le s\le3\sigma-1$, the term $\Big\lVert \left( \frac{\xi ^2\hat{f}}{\phi \left( \xi \right)} \right) ^{\lor} \Big\rVert _{Z^{s,\sigma -1}}$ is not necessary; if $j=1,4$, $0\le s\le3\sigma$ or $j=2,5$, $1\le s\le3\sigma+1$, then the term $\Big\lVert \left( \frac{\xi ^2\hat{f}}{\phi \left( \xi \right)} \right) ^{\lor} \Big\rVert _{Z^{s,\sigma -1}}$ is also not needed. The time-adapted Bourgain space, $Z^{s,\sigma-1}$, is introduced to cover all values of $s$.
	\end{proof}
	By \textbf{Lemma} \ref{lem 3.3}, \ref{lem 3.4}, \ref{lem 3.5}, \textbf{Theorem} \ref{thm1.3} and contraction mapping principal, we can obtain \textbf{Theorem} \ref{thm1.5}. 
	\subsection{Proof of Theorem \ref{thm1.3}, multilinear estimates in $X^{s,b}$ with $b=\frac{1}{2}$}\label{section4}
	To adjust for the proof, in the rest of this paper, we take $m$ being a fixed constant satisfying $m\gg1$.
	
	\begin{proof}
		First, we prove \eqref{1.9}. We only consider the case when $s$ is close to $-\frac{1}{2}$ since $\frac{\left< \xi _1 \right> \left< \xi _2 \right> \left< \xi _3 \right>}{\left< \xi _4 \right>}\gtrsim1$. Set $p=-s>0$, according to $\phi \left( \xi \right) =\left| \xi \right|\sqrt{\xi ^4+\xi ^2+1}\sim\left| \xi \right|\left< \xi \right> ^2$, one has,
		$$\Bigg\lVert \left( \frac{\xi ^2\widehat{uvw}}{\phi \left( \xi \right)} \right) ^{\lor} \Bigg\rVert _{X^{s,\sigma -1}}\sim \Bigg\lVert \frac{\left| \xi \right|\widehat{uvw}}{\left< \xi \right> ^{p+2}\left< \left| \tau \right|-\left| \xi \right|^3+\frac{1}{2}\left| \xi \right| \right> ^{1-\sigma}} \Bigg\rVert _{L_{\xi ,\tau}^{2}}.$$
		Let us denote $L:=\left| \tau \right|-\left| \xi \right|^3+\frac{1}{2}\left| \xi \right|$, $M:=\left< L \right> ^{\frac{1}{2}}+\chi _{\left| \xi \right|\le 1}\left< L \right> ^\sigma$, then, based on duality and Plancherel’s equality, to prove \eqref{1.9}, it suffices to prove:
		\begin{equation}\label{4.1}
		\left| \int_{A_4}{\frac{\left< \xi _1 \right> ^p\left< \xi _2 \right> ^p\left< \xi _3 \right> ^p\prod_{j=1}^4{f_j\left( \xi _j,\tau _j \right)}}{\left< \xi _4 \right> ^{p+1}M_1M_2M_3\left< L_4 \right> ^{1-\sigma}}} \right|\lesssim\prod_{j=1}^4{\lVert f_j \rVert _{L_{\xi}^{2}L_{\tau}^{2}}},
		\end{equation}
		where, \begin{equation*}
		f_1=\left< \xi \right> ^sM_1\widehat{u}\left( \xi ,\tau \right) ,\quad f_2=\left< \xi \right> ^sM_2\widehat{v}\left( \xi ,\tau \right) ,\quad f_3=\left< \xi \right> ^sM_3\widehat{w}\left( \xi ,\tau \right),\quad f_j\in L^2\left( \mathbb{R}^2 \right),\quad j=1,2,3,4
		\end{equation*}and,
		\begin{equation*}
		L_j=\tau _j-sgn \left( \tau _j \right) \left( \left| \xi _j \right|^3-\frac{1}{2}\left| \xi _j \right| \right),\quad j=1,2,3,4,\quad M_j:=\left< L_j \right> ^{\frac{1}{2}}+\chi _{\left| \xi _j \right|\le 1}\left< L_j \right> ^\sigma,\quad j=1,2,3.
		\end{equation*}
		In addition, we introduce the resonance function,
		$$H\left( \xi _1,\xi _2,\xi _3,\xi _4 \right) =\sum_{j=1}^4{L_j},$$
		which is the key to handle the estimate in high frequencies part. Without loss of generality, we assume $\frac{1}{3}<p<\frac{1}{2}$ and take $\sigma_0=\min \left\{ \frac{7}{12},1-p \right\} $. Moreover, by the symmetry construction between $\xi_{1}$, $\xi_{2}$, $\xi_{3}$ in \eqref{4.1}, there is no harm to assume $\left| \xi _1 \right|\le \left| \xi _2 \right|\le \left| \xi _3 \right|$ which impies that $\left| \xi _4 \right|\le \left| \xi _1 \right|+\left| \xi _2 \right|+\left| \xi _3 \right|\lesssim\left| \xi _3 \right|$. In order to establish the estimate \eqref{4.1}, we split the hyperplane into the following pieces:\begin{itemize}
			\item \textbf{Case \uppercase\expandafter{A}.}\quad$\left| \xi _3 \right|\le4m,$
			\item \textbf{Case \uppercase\expandafter{B}.}\quad$\left| \xi _3 \right|> 4m$, $\left| \xi _2 \right|\le m,$ 
			\item \textbf{Case \uppercase\expandafter{C}.}\quad$\left| \xi _3 \right|>4m,\ \left| \xi _2 \right|>m,\ \left| \xi _1 \right|\le 1,$
			\item \textbf{Case \uppercase\expandafter{D}.}\quad$\left| \xi _3 \right|>4m,\ \left| \xi _2 \right|>m,\ \left| \xi _1 \right|> 1$, $\left| \xi _4 \right|\sim\left| \xi _3 \right|,$
			\item \textbf{Case \uppercase\expandafter{E}.}\quad$\left| \xi _3 \right|>4m,\ \left| \xi _2 \right|>m,\ \left| \xi _1 \right|> 1$, $\left| \xi _1 \right|\lesssim\left| \xi _4 \right|\ll\left| \xi _2 \right|\sim\left| \xi _3 \right|,$
			\item \textbf{Case \uppercase\expandafter{F}.}\quad$\left| \xi _3 \right|>4m,\ \left| \xi _2 \right|>m,\ \left| \xi _1 \right|> 1$, $\left| \xi _4 \right|\ll\left| \xi _1 \right|$, $\left| \xi _2 \right|\sim\left| \xi _3 \right|,$
		\end{itemize}
		and denote the integral in each domain as $\varGamma_j$ with $j=1,...,6$, accordingly. We write,
		\begin{equation*}
		J:=\frac{\left< \xi _1 \right> ^p\left< \xi _2 \right> ^p\left< \xi _3 \right> ^p}{\left< \xi _4 \right> ^{p+1}},\quad \quad K:=\frac{1}{M_1M_2M_3\left< L_4 \right> ^{1-\sigma}}.
		\end{equation*}
		Next, we will prove that $\varGamma_ j$ is bounded by $\prod_{j=1}^4{\lVert f_j \rVert _{L_{\xi}^{2}L_{\tau}^{2}}}$.
		\begin{itemize}
			\item \textbf{Case \uppercase\expandafter{A}.} \quad $\left| \xi _3 \right|\le4m.$\\
			It follows that $\left| \xi _1 \right|$, $\left| \xi _2 \right|\lesssim1$, thus, $J\lesssim1$, $K\lesssim\frac{1}{\left< L_1 \right> ^{\frac{1}{2}}\left< L_2 \right> ^{\frac{1}{2}}\left< L_3 \right> ^{\frac{1}{2}}\left< L_4 \right> ^{1-\sigma}}$. Hence, one has,
			\begin{equation*}
			\begin{aligned}
			\varGamma _1\lesssim&\int_{A_4}{\left| \frac{f_1\left( \xi _1,\tau _1 \right)}{\left< L_1 \right> ^{\frac{1}{2}}}\frac{f_2\left( \xi _2,\tau _2 \right)}{\left< L_2 \right> ^{\frac{1}{2}}}\frac{f_3\left( \xi _3,\tau _3 \right)}{\left< L_3 \right> ^{\frac{1}{2}}}\frac{f_4\left( \xi _4,\tau _4 \right)}{\left< L_4 \right> ^{1-\sigma}} \right|}\\
			\lesssim&\lVert F_{1-\sigma}^{(4)} \rVert _{L_{x}^{2}L_{t}^{2}}\prod_{j=1}^3{\lVert F_{\frac{1}{2}}^{(j)} \rVert _{L_{x}^{6}L_{t}^{6}}}\lesssim\prod_{j=1}^4{\lVert f_j \rVert _{L_{\xi}^{2}L_{\tau}^{2}}}
			\end{aligned}
			\end{equation*}
			by \textbf{Lemma} \ref{lem1.7}, \ref{lem2.10}, where $\hat{F}_{\rho}^{\left( j \right)}:=\frac{f_j\left( \xi _j,\tau _j \right)}{\left< L_j \right> ^{\rho}}.$
			\item \textbf{Case \uppercase\expandafter{B}.}\quad$\left| \xi _3 \right|> 4m$, $\left| \xi _2 \right|\le m,$\\
			It follows that $\left| \xi _1 \right|\le \left| \xi _2 \right|\le m$, $\left| \xi _4 \right|>2m$, $J\lesssim\left< \xi _3 \right> ^p\lesssim\left| \xi _3 \right|^{\frac{1}{2}}\lesssim\left| \xi _3 \right|$ for $p<\frac{1}{2}$. \\
			\subitem \textbf{Case B.(a). }
			$\left< L_1 \right> =MAX:=\underset{j=1,2,3,4}{\max}\left< L_j \right> $, then it follows that, $$\left< L_1 \right> ^{\frac{1}{2}}\left< L_2 \right> ^{\frac{1}{2}}\left< L_3 \right> ^{\frac{1}{2}}\left< L_4 \right> ^{1-\sigma}\ge \left< L_1 \right> ^{\frac{5}{12}}\left< L_2 \right> ^{\frac{13}{24}}\left< L_3 \right> ^{\frac{13}{24}}\left< L_4 \right> ^{1-\sigma}.$$ Hence, one has,
			\begin{equation*}
			\begin{aligned}
			\varGamma _2\lesssim&\int_{A _4}{\left| \frac{f_1\left( \xi _1,\tau _1 \right)}{\left< L_1 \right> ^{\frac{5}{12}}}\frac{\chi _{\left| \xi _2 \right|\le m}f_2\left( \xi _2,\tau _2 \right)}{\left< L_2 \right> ^{\frac{13}{24}}}\frac{\left| \xi _3 \right|\chi _{\left| \xi _3 \right|>4m}f_3\left( \xi _3,\tau _3 \right)}{\left< L_3 \right> ^{\frac{13}{24}}}\frac{f_4\left( \xi _4,\tau _4 \right)}{\left< L_4 \right> ^{1-\sigma}} \right|}\\
			\lesssim&\lVert F_{\frac{5}{12}}^{(1)} \rVert _{L_{x}^{4}L_{t}^{4}}\lVert P_{m}F_{\frac{13}{24}}^{(2)} \rVert _{L_{x}^{2}L_{t}^{\infty}}\lVert D_xP^{4m}F_{\frac{13}{24}}^{(3)} \rVert _{L_{x}^{\infty}L_{t}^{2}}\lVert F_{1-\sigma}^{(4)} \rVert _{L_{x}^{4}L_{t}^{4}}\lesssim\prod_{j=1}^4{\lVert f_j \rVert _{L_{\xi}^{2}L_{\tau}^{2}}}
			\end{aligned}
			\end{equation*}
			by \textbf{Lemma} \ref{lem1.7}, \ref{lem2.8}, \ref{lem2.9}, \ref{2.10} since $\sigma<\frac{2}{3}$.
			\subitem \textbf{Case B.(b). }
			$\left< L_2 \right> =MAX $, then it follows that, $$\left< L_1 \right> ^{\frac{1}{2}}\left< L_2 \right> ^{\frac{1}{2}}\left< L_3 \right> ^{\frac{1}{2}}\left< L_4 \right> ^{1-\sigma}\ge \left< L_1 \right> ^{\frac{13}{24}}\left< L_2 \right> ^{\frac{5}{12}}\left< L_3 \right> ^{\frac{13}{24}}\left< L_4 \right> ^{1-\sigma}.$$
			Hence, similar to \textbf{Case B.(a)}, the integral in this domain is bounded by:
			\begin{equation*}
			\varGamma _2\lesssim\lVert P_{m}F_{\frac{13}{24}}^{(1)} \rVert _{L_{x}^{2}L_{t}^{\infty}}\lVert F_{\frac{5}{12}}^{(2)} \rVert _{L_{x}^{4}L_{t}^{4}}\lVert D_xP^{4m}F_{\frac{13}{24}}^{(3)} \rVert _{L_{x}^{\infty}L_{t}^{2}}\lVert F_{1-\sigma}^{(4)} \rVert _{L_{x}^{4}L_{t}^{4}}\lesssim \prod_{j=1}^4{\lVert f_j \rVert _{L_{\xi}^{2}L_{\tau}^{2}}}.
			\end{equation*}
			\subitem \textbf{Case B.(c). }
			$\left< L_3 \right> =MAX $, then it follows that, $$\left< L_1 \right> ^{\frac{1}{2}}\left< L_2 \right> ^{\frac{1}{2}}\left< L_3 \right> ^{\frac{1}{2}}\left< L_4 \right> ^{1-\sigma}\ge \left< L_1 \right> ^{\frac{1}{2}}\left< L_2 \right> ^{\frac{1}{2}}\left< L_3 \right> ^{\frac{1}{3}}\left< L_4 \right> ^{\frac{7}{6}-\sigma},\quad J\lesssim \frac{\left< \xi _3 \right> ^p}{\left< \xi _4 \right> ^{p+1}}\lesssim \frac{\left| \xi _3 \right|^{\frac{1}{2}}}{\left| \xi _4 \right|^{\frac{1}{4}}}.$$
			Hence, one has,
			\begin{equation*}
			\begin{aligned}
			\varGamma _2\lesssim&\int_{A _4}{\left| \frac{f_1\left( \xi _1,\tau _1 \right)}{\left< L_1 \right> ^{\frac{1}{2}}}\frac{\chi _{\left| \xi _2 \right|\le m}f_2\left( \xi _2,\tau _2 \right)}{\left< L_2 \right> ^{\frac{1}{2}}}\frac{\left| \xi _3 \right|^{\frac{1}{2}}\chi _{\left| \xi _3 \right|>4m}f_3\left( \xi _3,\tau _3 \right)}{\left< L_3 \right> ^{\frac{1}{3}}}\frac{\chi _{\left| \xi _4 \right|\le 2m}f_4\left( \xi _4,\tau _4 \right)}{\left| \xi _4 \right|^{\frac{1}{4}}\left< L_4 \right> ^{\frac{7}{6}-\sigma}} \right|}\\
			\lesssim&\lVert F_{\frac{1}{2}}^{(1)} \rVert _{L_{x}^{4}L_{t}^{4}}\lVert F_{\frac{1}{2}}^{(2)} \rVert _{L_{x}^{4}L_{t}^{4}}\lVert D_{x}^{\frac{1}{2}}P^{4m}F_{\frac{1}{3}}^{(3)} \rVert _{L_{x}^{4}L_{t}^{2}}\lVert D_{x}^{-\frac{1}{4}}P^{2m}F_{\frac{7}{6}-\sigma}^{(4)} \rVert _{L_{x}^{4}L_{t}^{\infty}}\lesssim\prod_{j=1}^4{\lVert f_j \rVert _{L_{\xi}^{2}L_{\tau}^{2}}}
			\end{aligned}
			\end{equation*}
			by \textbf{Lemma} \ref{lem1.7}, \ref{lem2.9}, \ref{lem2.10}, \ref{lem2.11}.
			\subitem \textbf{Case B.(d). }
			If $\left< L_4 \right> =MAX $, then it follows that, $$\left< L_1 \right> ^{\frac{1}{2}}\left< L_2 \right> ^{\frac{1}{2}}\left< L_3 \right> ^{\frac{1}{2}}\left< L_4 \right> ^{1-\sigma}\ge \left< L_1 \right> ^{\frac{1}{2}}\left< L_2 \right> ^{\frac{13}{24}}\left< L_3 \right> ^{\frac{13}{24}}\left< L_4 \right> ^{\frac{11}{12}-\sigma}.$$
			Hence, the integral in this domain is bounded by:
			\begin{equation*}
			\begin{aligned}
			\varGamma _2\lesssim&\int_{A _4}{\left| \frac{f_1\left( \xi _1,\tau _1 \right)}{\left< L_1 \right> ^{\frac{1}{2}}}\frac{\chi _{\left| \xi _2 \right|\le m}f_2\left( \xi _2,\tau _2 \right)}{\left< L_2 \right> ^{\frac{13}{24}}}\frac{\left| \xi _3 \right|\chi _{\left| \xi _3 \right|>4m}f_3\left( \xi _3,\tau _3 \right)}{\left< L_3 \right> ^{\frac{13}{24}}}\frac{f_4\left( \xi _4,\tau _4 \right)}{\left< L_4 \right> ^{\frac{11}{12}-\sigma}} \right|}\\
			\lesssim&\lVert F_{\frac{1}{2}}^{(1)} \rVert _{L_{x}^{4}L_{t}^{4}}\lVert P_{m}F_{\frac{13}{24}}^{(2)} \rVert _{L_{x}^{2}L_{t}^{\infty}}\lVert D_xP^{4m}F_{\frac{13}{12}}^{(3)} \rVert _{L_{x}^{\infty}L_{t}^{2}}\lVert F_{\frac{11}{12}-\sigma}^{(4)} \rVert _{L_{x}^{4}L_{t}^{4}}\lesssim\prod_{j=1}^4{\lVert f_j \rVert _{L_{\xi}^{2}L_{\tau}^{2}}}.
			\end{aligned}
			\end{equation*}
			\item \textbf{Case C.}\quad$\left| \xi _3 \right|>4m,\ \left| \xi _2 \right|>m,\ \left| \xi _1 \right|\le 1,$\\
			It follows that $J\lesssim \left< \xi _2 \right> ^p\left< \xi _3 \right> ^p\lesssim \left| \xi _3 \right|$ for $p<\frac{1}{2}$, $K\lesssim\frac{1}{\left< L_1 \right> ^\sigma\left< L_2 \right> ^{\frac{1}{2}}\left< L_3 \right> ^{\frac{1}{2}}\left< L_4 \right> ^{1-\sigma}}$. 
			\subitem \textbf{Case C.(a).}\quad$\left< L_1 \right> =MAX$, then it follows that, $$\left< L_1 \right> ^\sigma\left< L_2 \right> ^{\frac{1}{2}}\left< L_3 \right> ^{\frac{1}{2}}\left< L_4 \right> ^{1-\sigma}\ge \left< L_1 \right> ^{\frac{\sigma}{2}+\frac{1}{4}}\left< L_2 \right> ^{\frac{1}{2}}\left< L_3 \right> ^{\frac{\sigma}{2}+\frac{1}{4}}\left< L_4 \right> ^{1-\sigma}.$$ Hence, the integral in this domain is bounded by:
			\begin{equation*}
			\begin{aligned}
			\varGamma _3\lesssim&\int_{A _4}{\left| \frac{\chi _{\left| \xi _1 \right|\le 1}f_1\left( \xi _1,\tau _1 \right)}{\left< L_1 \right> ^{\frac{\sigma}{2}+\frac{1}{4}}}\frac{f_2\left( \xi _2,\tau _2 \right)}{\left< L_2 \right> ^{\frac{1}{2}}}\frac{\left| \xi _3 \right|\chi _{\left| \xi _3 \right|>4m}f_3\left( \xi _3,\tau _3 \right)}{\left< L_3 \right> ^{\frac{\sigma}{2}+\frac{1}{4}}}\frac{f_4\left( \xi _4,\tau _4 \right)}{\left< L_4 \right> ^{1-\sigma}} \right|}\\
			\lesssim&\lVert P_{1}F_{\frac{\sigma}{2}+\frac{1}{4}}^{(1)} \rVert _{L_{x}^{2}L_{t}^{\infty}}\lVert F_{\frac{1}{2}}^{(2)} \rVert _{L_{x}^{4}L_{t}^{4}}\lVert D_xP^{4m}F_{\frac{\sigma}{2}+\frac{1}{4}}^{(3)} \rVert _{L_{x}^{\infty}L_{t}^{2}}\lVert F_{1-\sigma}^{(4)} \rVert _{L_{x}^{4}L_{t}^{4}}\lesssim\prod_{j=1}^4{\lVert f_j \rVert _{L_{\xi}^{2}L_{\tau}^{2}}}.
			\end{aligned}
			\end{equation*}
			\subitem \textbf{Case C.(b).}\quad$\left< L_2 \right> =MAX $, then it follows that, $$\left< L_1 \right> ^\sigma\left< L_2 \right> ^{\frac{1}{2}}\left< L_3 \right> ^{\frac{1}{2}}\left< L_4 \right> ^{1-\sigma}\ge \left< L_1 \right> ^\sigma\left< L_2 \right> ^{\frac{5}{12}}\left< L_3 \right> ^{\frac{7}{12}}\left< L_4 \right> ^{1-\sigma}.$$
			Hence, the integral in this domain is bounded by:
			\begin{equation*}
			\begin{aligned}
			\varGamma _3\lesssim&\int_{A _4}{\left| \frac{\chi _{\left| \xi _1 \right|\le 1}f_1\left( \xi _1,\tau _1 \right)}{\left< L_1 \right> ^{\sigma}}\frac{f_2\left( \xi _2,\tau _2 \right)}{\left< L_2 \right> ^{\frac{5}{6}}}\frac{\left| \xi _3 \right|\chi _{\left| \xi _3 \right|>4m}f_3\left( \xi _3,\tau _3 \right)}{\left< L_3 \right> ^{\frac{7}{6}}}\frac{f_4\left( \xi _4,\tau _4 \right)}{\left< L_4 \right> ^{1-\sigma}} \right|}\\
			\lesssim&\lVert P_1F_{\sigma}^{(1)} \rVert _{L_{x}^{2}L_{t}^{\infty}}\lVert F_{\frac{5}{12}}^{(2)} \rVert _{L_{x}^{4}L_{t}^{4}}\lVert D_xP^{4m}F_{\frac{7}{12}}^{(3)} \rVert _{L_{x}^{\infty}L_{t}^{2}}\lVert F_{1-\sigma}^{(4)} \rVert _{L_{x}^{4}L_{t}^{4}}\lesssim\prod_{j=1}^4{\lVert f_j \rVert _{L_{\xi}^{2}L_{\tau}^{2}}}.
			\end{aligned}
			\end{equation*}
			\subitem \textbf{Case C.(c).}\quad$\left< L_3 \right> =MAX.$
			\subsubitem \textbf{Case C.(c).1.}\quad$\left| \xi _4 \right|\gtrsim\left| \xi _3 \right|$. It follows that
			$J\lesssim\frac{\left< \xi _2 \right> ^p\left< \xi _3 \right> ^p}{\left< \xi _3 \right> ^{p+1}}\lesssim1 .$
			Hence, the integral in this domain is bounded by $\prod_{j=1}^4{\lVert f_j \rVert _{L_{\xi}^{2}L_{\tau}^{2}}}$ using the argument similar to \textbf{Case A}.
			\subsubitem \textbf{Case C.(c).2.}\quad$\left| \xi _4 \right|\ll\left| \xi _3 \right|$. First, we point out that $\left| \xi _3 \right|\sim\left| \xi _2 \right|$. Thus, $J\lesssim\left| \xi _3 \right|\sim\left| \xi _2 \right|$, $\left< L_1 \right> ^\sigma\left< L_2 \right> ^{\frac{1}{2}}\left< L_3 \right> ^{\frac{1}{2}}\left< L_4 \right> ^{1-\sigma}\ge \left< L_1 \right> ^\sigma\left< L_2 \right> ^{\frac{7}{12}}\left< L_3 \right> ^{\frac{5}{12}}\left< L_4 \right> ^{1-\sigma}$. Hence, the integral in this domain is bounded by:
			\begin{equation*}
			\begin{aligned}
			\varGamma _3\lesssim&\int_{A _4}{\left| \frac{\chi _{\left| \xi _1 \right|\le 1}f_1\left( \xi _1,\tau _1 \right)}{\left< L_1 \right> ^{\sigma}}\frac{\chi _{\left| \xi _2 \right|>m}\left| \xi _2 \right|f_2\left( \xi _2,\tau _2 \right)}{\left< L_2 \right> ^{\frac{7}{6}}}\frac{f_3\left( \xi _3,\tau _3 \right)}{\left< L_3 \right> ^{\frac{5}{6}}}\frac{f_4\left( \xi _4,\tau _4 \right)}{\left< L_4 \right> ^{1-\sigma}} \right|}\\
			\lesssim&\lVert P_1F_{\sigma}^{(1)} \rVert _{L_{x}^{2}L_{t}^{\infty}}\lVert D_xP^{m}F_{\frac{7}{12}}^{(2)} \rVert _{L_{x}^{\infty}L_{t}^{2}}\lVert F_{\frac{5}{12}}^{(3)} \rVert _{L_{x}^{4}L_{t}^{4}}\lVert F_{1-\sigma}^{(4)} \rVert _{L_{x}^{4}L_{t}^{4}}\lesssim\prod_{j=1}^4{\lVert f_j \rVert _{L_{\xi}^{2}L_{\tau}^{2}}}.
			\end{aligned}
			\end{equation*}
			\subitem \textbf{Case C.(d).}\quad$\left< L_4 \right> =MAX $, then it follows that, $$\left< L_1 \right> ^\sigma\left< L_2 \right> ^{\frac{1}{2}}\left< L_3 \right> ^{\frac{1}{2}}\left< L_4 \right> ^{1-\sigma}\ge \left< L_1 \right> ^\sigma\left< L_2 \right> ^{\frac{1}{2}}\left< L_3 \right> ^{\frac{7}{12}}\left< L_4 \right> ^{\frac{11}{12}-\sigma}.$$
			Hence, the integral in this domain is bounded by:
			\begin{equation*}
			\begin{aligned}
			\varGamma _3\lesssim&\int_{A _4}{\left| \frac{\chi _{\left| \xi _1 \right|\le 1}f_1\left( \xi _1,\tau _1 \right)}{\left< L_1 \right> ^\sigma}\frac{f_2\left( \xi _2,\tau _2 \right)}{\left< L_2 \right> ^{\frac{1}{2}}}\frac{\left| \xi _3 \right|\chi _{\left| \xi _3 \right|>4m}f_3\left( \xi _3,\tau _3 \right)}{\left< L_3 \right> ^{\frac{7}{12}}}\frac{f_4\left( \xi _4,\tau _4 \right)}{\left< L_4 \right> ^{\frac{11}{12}-\sigma}} \right|}\\
			\lesssim&\lVert P_{1}F_{\sigma}^{(1)} \rVert _{L_{x}^{2}L_{t}^{\infty}}\lVert F_{\frac{1}{2}}^{(2)} \rVert _{L_{x}^{4}L_{t}^{4}}\lVert D_xP^{4m}F_{\frac{7}{12}}^{(3)} \rVert _{L_{x}^{\infty}L_{t}^{2}}\lVert F_{\frac{11}{12}-\sigma}^{(4)} \rVert _{L_{x}^{4}L_{t}^{4}}\lesssim\prod_{j=1}^4{\lVert f_j \rVert _{L_{\xi}^{2}L_{\tau}^{2}}}.
			\end{aligned}
			\end{equation*}
			\item \textbf{Case D.}\quad$\left| \xi _3 \right|>4m,\ \left| \xi _2 \right|>m,\ \left| \xi _1 \right|> 1$, $\left| \xi _4 \right|\sim\left| \xi _3 \right|.$\\
			It follows that $J\lesssim\frac{\left< \xi _3 \right> ^p\left< \xi _3 \right> ^p\left< \xi _3 \right> ^p}{\left< \xi _3 \right> ^{p+1}}\lesssim1$ since $\left| \xi _1 \right|\le \left| \xi _2 \right|\le \left| \xi _3 \right|$ and $p<\frac{1}{2}$. Then, by arguement similar to \textbf{Case A}, we can obtain that $\varGamma _4\lesssim\prod_{j=1}^4{\lVert f_j \rVert _{L_{\xi}^{2}L_{\tau}^{2}}}.$
			\item \textbf{Case E.}\quad$\left| \xi _3 \right|>4m,\ \left| \xi _2 \right|>m,\ \left| \xi _1 \right|> 1$, $\left| \xi _1 \right|\lesssim\left| \xi _4 \right|\ll\left| \xi _2 \right|\sim\left| \xi _3 \right|.$\\
			
			It follows that $K\le\frac{1}{\left< L_1 \right> ^{\frac{1}{2}}\left< L_2 \right> ^{\frac{1}{2}}\left< L_3 \right> ^{\frac{1}{2}}\left< L_4 \right> ^{1-\sigma}}$. 
			\subitem \textbf{Case E.(a).}\quad$\left< L_1 \right> =MAX$, then one has, $$\left< L_1 \right> ^{\frac{1}{2}}\left< L_2 \right> ^{\frac{1}{2}}\left< L_3 \right> ^{\frac{1}{2}}\left< L_4 \right> ^{1-\sigma}\ge \left< L_1 \right> ^{\frac{9}{24}}\left< L_2 \right> ^{\frac{1}{2}}\left< L_3 \right> ^{\frac{13}{24}}\left< L_4 \right> ^{\frac{13}{12}-\sigma},\quad J\lesssim \frac{\left< \xi _2 \right> ^p\left< \xi _3 \right> ^p}{\left< \xi _4 \right>}\lesssim \frac{\left| \xi _3 \right|}{\left< \xi _4 \right>}.$$ 
			Hence, one has,
			\begin{equation*}
			\begin{aligned}
			\varGamma _5\lesssim&\int_{A_4}{\left| \frac{f_1\left( \xi _1,\tau _1 \right)}{\left< L_1 \right> ^{\frac{9}{24}}}\frac{f_2\left( \xi _2,\tau _2 \right)}{\left< L_2 \right> ^{\frac{1}{2}}}\frac{\left| \xi _3 \right|\chi _{\left| \xi _3 \right|\ge 4m}f_3\left( \xi _3,\tau _3 \right)}{\left< L_3 \right> ^{\frac{13}{24}}}\frac{f_4\left( \xi _4,\tau _4 \right)}{\left< \xi _4 \right> \left< L_4 \right> ^{\frac{13}{12}-\sigma}} \right|}\\
			\lesssim&\lVert F_{\frac{9}{24}}^{(1)} \rVert _{L_{x}^{4}L_{t}^{4}}\lVert F_{\frac{1}{2}}^{(2)} \rVert _{L_{x}^{4}L_{t}^{4}}\lVert D_xP^{4m}F_{\frac{13}{24}}^{(3)} \rVert _{L_{x}^{\infty}L_{t}^{2}} \lVert J_{x}^{-1}F_{\frac{13}{12}-\sigma}^{(4)} \rVert _{L_{x}^{2}L_{t}^{\infty}}\lesssim\prod_{j=1}^4{\lVert f_j \rVert _{L_{\xi}^{2}L_{\tau}^{2}}}
			\end{aligned}
			\end{equation*}
			by \textbf{Lemma} \ref{lem1.7}, \ref{lem2.9}, \ref{lem2.10}, since $\sigma<\frac{7}{12}$.
			\subitem \textbf{Case E.(b).}\quad$\left< L_2 \right> =MAX $, then it follows that, $$\left< L_1 \right> ^{\frac{1}{2}}\left< L_2 \right> ^{\frac{1}{2}}\left< L_3 \right> ^{\frac{1}{2}}\left< L_4 \right> ^{1-\sigma}\ge \left< L_1 \right> ^{\frac{13}{24}}\left< L_2 \right> ^{\frac{5}{12}}\left< L_3 \right> ^{\frac{13}{24}}\left< L_4 \right> ^{1-\sigma},\quad J\lesssim\frac{\left< \xi _1 \right> ^p\left< \xi _3 \right> ^{2p}}{\left< \xi _1 \right> ^{p+1}}\lesssim\frac{\left< \xi _3 \right>}{\left< \xi _1 \right>}
			\sim\frac{\left| \xi _3 \right|}{\left< \xi _1 \right>}.$$
			Hence, similar to \textbf{Case E.(a)}, the integral in this domain is bounded by:
			\begin{equation*}
			\varGamma _5\lesssim\lVert J_{x}^{-1}F_{\frac{13}{24}}^{(1)} \rVert _{L_{x}^{2}L_{t}^{\infty}}\lVert F_{\frac{5}{12}}^{(2)} \rVert _{L_{x}^{4}L_{t}^{4}}\lVert D_xP^{4m}F_{\frac{13}{24}}^{(3)} \rVert _{L_{x}^{\infty}L_{t}^{2}}\lVert F_{1-\sigma}^{(4)} \rVert _{L_{x}^{4}L_{t}^{4}}\lesssim \prod_{j=1}^4{\lVert f_j \rVert _{L_{\xi}^{2}L_{\tau}^{2}}}.
			\end{equation*}
			\subitem \textbf{Case E.(c).}\quad$\left< L_3 \right> =MAX $, then it follows that, $$\left< L_1 \right> ^{\frac{1}{2}}\left< L_2 \right> ^{\frac{1}{2}}\left< L_3 \right> ^{\frac{1}{2}}\left< L_4 \right> ^{1-\sigma}\ge \left< L_1 \right> ^{\frac{13}{24}}\left< L_2 \right> ^{\frac{13}{24}}\left< L_3 \right> ^{\frac{5}{12}}\left< L_4 \right> ^{1-\sigma},\quad J\lesssim\frac{\left| \xi _3 \right|}{\left< \xi _1 \right>}\sim \frac{\left| \xi _2 \right|}{\left< \xi _1 \right>}.$$
			Hence, the integral in this domain is bounded by:
			\begin{equation*}
			\varGamma _5\lesssim\lVert J_{x}^{-1}F_{\frac{13}{24}}^{(1)} \rVert _{L_{x}^{2}L_{t}^{\infty}}\lVert D_xP^mF_{\frac{13}{24}}^{(3)} \rVert _{L_{x}^{\infty}L_{t}^{2}}\lVert F_{\frac{5}{12}}^{(3)} \rVert _{L_{x}^{4}L_{t}^{4}}\lVert F_{1-b}^{(4)} \rVert _{L_{x}^{4}L_{t}^{4}}\lesssim \prod_{j=1}^4{\lVert f_j \rVert _{L_{\xi}^{2}L_{\tau}^{2}}}.
			\end{equation*}
			\subitem \textbf{Case E.(d).}\quad$\left< L_4 \right> =MAX $, then it follows that, $$\left< L_1 \right> ^{\frac{1}{2}}\left< L_2 \right> ^{\frac{1}{2}}\left< L_3 \right> ^{\frac{1}{2}}\left< L_4 \right> ^{1-\sigma}\ge \left< L_1 \right> ^{\frac{13}{24}}\left< L_2 \right> ^{\frac{13}{12}-\sigma}\left< L_3 \right> ^{\frac{1}{2}}\left< L_4 \right> ^{\frac{9}{24}},\quad J\lesssim\frac{\left| \xi _2 \right|}{\left< \xi _1 \right>}.$$
			Hence, the integral in this domain is bounded by:
			\begin{equation*}
			\varGamma _5\lesssim\lVert J_{x}^{-1}F_{\frac{13}{24}}^{(1)} \rVert _{L_{x}^{2}L_{t}^{\infty}}\lVert D_xP^mF_{\frac{13}{12}-\sigma}^{(2)} \rVert _{L_{x}^{\infty}L_{t}^{2}}\lVert F_{\frac{1}{2}}^{(3)} \rVert _{L_{x}^{4}L_{t}^{4}}\lVert F_{\frac{9}{24}}^{(4)} \rVert _{L_{x}^{4}L_{t}^{4}}\lesssim \prod_{j=1}^4{\lVert f_j \rVert _{L_{\xi}^{2}L_{\tau}^{2}}}.
			\end{equation*}
			\item \textbf{Case F.}\quad$\left| \xi _3 \right|>4m,\ \left| \xi _2 \right|>m,\ \left| \xi _1 \right|> 1$, $\left| \xi _4 \right|\ll\left| \xi _1 \right|$, $\left| \xi _2 \right|\sim\left| \xi _3 \right|,$\\
			For $m\gg1$, we claim that in this whole domain, there exists a uniform bound of the resonance function $H$:
			\begin{equation*}
			\left| H \right|\gtrsim\left| \xi _1\xi _2\xi _3 \right|.
			\end{equation*}
			This claim follows by direct but tedious computation, for the details, reader can refer \textbf{Lemma} \ref{lem5.1} in \textbf{Appendix}. Moreover, we establish the estimate in the following subcases:
			$$\textbf{Case F.(a).} \left< L_1 \right> =MAX,\quad\quad\quad
			\textbf{Case F.(b).} \left< L_2 \right> =MAX,$$
			$$\textbf{Case F.(c).} \left< L_3 \right> =MAX,\quad\quad\quad
			\textbf{Case F.(d).} \left< L_4 \right> =MAX.$$
			If \textbf{Case F.(a)} holds, it then follows that $\left< L_1 \right> ^{\sigma}\left< L_2 \right> ^{\sigma}\left< L_3 \right> ^{\sigma}\left< L_4 \right> ^{1-\sigma}\ge\left< L_1 \right> ^{1-\sigma}\left< L_2 \right> ^{\sigma}\left< L_3 \right> ^{\sigma}\left< L_4 \right> ^{\sigma}$ and $\left< L_1 \right> ^{1-\sigma}\gtrsim\left< H \right> ^{1-\sigma}\gtrsim\left< \xi _1\xi _2\xi _3 \right> ^{1-\sigma}$, which will lead to $J\lesssim\frac{\left| \xi _1\xi _2\xi _3 \right|^p}{\left< \xi _1\xi _2\xi _3 \right> ^{1-\sigma}}\lesssim\left| \xi _1\xi _2\xi _3 \right|^{p+\sigma-1}\lesssim1$ since $\sigma\le \sigma_0\le1-p$. Hence,
			\begin{equation*}
			\begin{aligned}
			\varGamma _6\lesssim&\int{\left| f_1\left( \xi _1,\tau _1 \right) \frac{f_2\left( \xi _2,\tau _2 \right)}{\left< L_2 \right> ^{\sigma}}\frac{f_3\left( \xi _3,\tau _3 \right)}{\left< L_3 \right> ^{\sigma}}\frac{f_4\left( \xi _4,\tau _4 \right)}{\left< L_4 \right> ^{\sigma}} \right|}\\
			\lesssim&\lVert F_{0}^{(1)} \rVert _{L_{x}^{2}L_{t}^{2}}\prod_{j=2}^4{\lVert F_{\sigma}^{(j)} \rVert _{L_{x}^{6}L_{t}^{6}}}\lesssim\prod_{j=1}^4{\lVert f_j \rVert _{L_{\xi}^{2}L_{\tau}^{2}}},
			\end{aligned}
			\end{equation*}
			by \textbf{Lemma} \ref{lem1.7}, \ref{lem2.10}. Similar arguments can show the desired bound of the integral for case \textbf{F.(b)}, \textbf{F.(c)} and \textbf{F.(d)} holds.\end{itemize}
		This completes the proof of \eqref{1.9}. Next, we turn to prove \eqref{1.10}. Simialr as before, set $p=-s>0$, based on duality, to prove \eqref{1.10}, it suffices to prove:
		\be\label{4.2} 
		\left| \int_{A_4}{\frac{\left| \xi _4 \right|\left< \xi _1 \right> ^p\left< \xi _2 \right> ^p\left< \xi _3 \right> ^p\prod_{j=1}^4{f_j\left( \xi _j,\tau _j \right)}}{\left< \tau _4 \right> ^{\frac{p}{3}+\sigma -\frac{1}{2}}\left< \xi _4 \right> ^2M_1M_2M_3\left< L_4 \right> ^{1-\sigma}}} \right|\lesssim \prod_{j=1}^4{\lVert f_j \rVert _{L_{\xi}^{2}L_{\tau}^{2}}},
		\ee
		where $f_j$, $M_j$, $L_j$ have the same definition in \eqref{4.1}. By symmetry, we assume $\left| \xi _1 \right|\le \left| \xi _2 \right|\le \left| \xi _3 \right|$ and take $\sigma_0=\frac{2}{3}$. First we consider the case when $\left| \xi _4 \right|\lesssim1$ or $\left< \tau _4 \right>\gtrsim \left< \xi _4 \right> ^3$. When $\left| \xi _4 \right|\lesssim1$, then, 
		$$LHS\ of\ \eqref{4.2}\lesssim\left|\int_{A_4}{\frac{\left< \xi _1 \right> ^p\left< \xi _2 \right> ^p\left< \xi _3 \right> ^p\prod_{j=1}^4{f_j\left( \xi _j,\tau _j \right)}}{\left< \xi _4 \right> ^2M_1M_2M_3\left< L_4 \right> ^{1-\sigma}}}\right|\lesssim\left|\int_{A_4}{\frac{\left< \xi _1 \right> ^p\left< \xi _2 \right> ^p\left< \xi _3 \right> ^p\prod_{j=1}^4{f_j\left( \xi _j,\tau _j \right)}}{\left< \xi _4 \right> ^{p+1}M_1M_2M_3\left< L_4 \right> ^{1-\sigma}}}\right|.$$
		Also, when $\left< \tau _4 \right>\gtrsim \left< \xi _4 \right> ^3$, one has,
		$$LHS\ of\ \eqref{4.2}\lesssim\left|\int_{A_4}{\frac{\left< \xi _1 \right> ^p\left< \xi _2 \right> ^p\left< \xi _3 \right> ^p\prod_{j=1}^4{f_j\left( \xi _j,\tau _j \right)}}{\left< \xi _4 \right> ^{p+1+3\left( \sigma -\frac{1}{2} \right)}M_1M_2M_3\left< L_4 \right> ^{1-\sigma}}}\right|\lesssim\left|\int_{A_4}{\frac{\left< \xi _1 \right> ^p\left< \xi _2 \right> ^p\left< \xi _3 \right> ^p\prod_{j=1}^4{f_j\left( \xi _j,\tau _j \right)}}{\left< \xi _4 \right> ^{p+1}M_1M_2M_3\left< L_4 \right> ^{1-\sigma}}}\right|.$$
		Thus, when $\left| \xi _4 \right|\lesssim1$ or $\left< \tau _4 \right>\gtrsim \left< \xi _4 \right> ^3$, \eqref{4.2} holds as a corollary of \eqref{4.1}. Next, we consider the case when $\left| \xi _4 \right|\gtrsim1$, $\left< \tau _4 \right>\ll \left< \xi _4 \right> ^3$. In this case, one has, $$\left| \xi _4 \right|\sim\left< \xi _4 \right> \sim\left< L_4 \right> ^{\frac{1}{3}}.$$
		Therefore, 
		$$LHS\ of\ \eqref{4.2}\lesssim\left|\int_{A_4}{\frac{\left< \xi _1 \right> ^p\left< \xi _2 \right> ^p\left< \xi _3 \right> ^p\prod_{j=1}^4{f_j\left( \xi _j,\tau _j \right)}}{\left< \xi _4 \right> ^{1+3\left( 1-\sigma \right)}M_1M_2M_3}}\right|\lesssim\left|\int_{A_4}{\frac{\left< \xi _1 \right> ^p\left< \xi _2 \right> ^p\left< \xi _3 \right> ^p\prod_{j=1}^4{f_j\left( \xi _j,\tau _j \right)}}{\left< \xi _4 \right> ^2M_1M_2M_3}}\right|,$$
		since $\sigma\le\frac{2}{3}$, then it suffices to prove:
		$$\left| \int_{A_4}{\frac{\left< \xi _1 \right> ^p\left< \xi _2 \right> ^p\left< \xi _3 \right> ^p\prod_{j=1}^4{f_j\left( \xi _j,\tau _j \right)}}{\left< \xi _4 \right> ^2M_1M_2M_3}} \right|\lesssim \prod_{j=1}^4{\lVert f_j \rVert _{L_{\xi}^{2}L_{\tau}^{2}}}.$$
		In order to establish this estimate, we also split the hyperplane into \textbf{Case.A} through \textbf{Case.F}
		and denote the integral of each domain as $\varUpsilon _j$, accordingly. For simplicity in natations, we write:
		\begin{equation*}
		\mathcal{J}:=\frac{\left< \xi _1 \right> ^p\left< \xi _2 \right> ^p\left< \xi _3 \right> ^p}{\left< \xi _4 \right> ^2},\quad \quad \mathcal{K}:=\frac{1}{M_1M_2M_3}.
		\end{equation*}
		\begin{itemize}
			\item \textbf{Case A.} \quad $\left| \xi _3 \right|\le4m.$\\
			It follows that $\left| \xi _1 \right|$, $\left| \xi _2 \right|\lesssim1$, thus, $\mathcal{J}\lesssim1$, $\mathcal{K}\lesssim\frac{1}{\left< L_1 \right> ^{\frac{1}{2}}\left< L_2 \right> ^{\frac{1}{2}}\left< L_3 \right> ^{\frac{1}{2}}}$. Hence, one has,
			\begin{equation*}
			\begin{aligned}
			\varUpsilon_1\lesssim&\int_{A_4}{\left| \frac{f_1\left( \xi _1,\tau _1 \right)}{\left< L_1 \right> ^{\frac{1}{2}}}\frac{f_2\left( \xi _2,\tau _2 \right)}{\left< L_2 \right> ^{\frac{1}{2}}}\frac{f_3\left( \xi _3,\tau _3 \right)}{\left< L_3 \right> ^{\frac{1}{2}}}f_4\left( \xi _4,\tau _4 \right) \right|}\\
			=&\lesssim\lVert F_{0}^{(4)} \rVert _{L_{x}^{2}L_{t}^{2}}\prod_{j=1}^3{\lVert F_{\frac{1}{2}}^{(j)} \rVert _{L_{x}^{6}L_{t}^{6}}}\lesssim\prod_{j=1}^4{\lVert f_j \rVert _{L_{\xi}^{2}L_{\tau}^{2}}}.
			\end{aligned}
			\end{equation*}
			\item \textbf{Case B.}\quad$\left| \xi _3 \right|> 4m$, $\left| \xi _2 \right|\le m,$\\
			It follows that $\left| \xi _1 \right|\le \left| \xi _2 \right|\le m$, $\left| \xi _4 \right|>2m$ and $\mathcal{J}\lesssim \frac{\left< \xi _3 \right> ^p}{\left< \xi _4 \right> ^2}\lesssim \frac{\left| \xi _3 \right|^{\frac{1}{2}}}{\left| \xi _4 \right|^{\frac{1}{4}}\left< L_4 \right> ^{\frac{7}{12}}}.$
			Hence, one has,
			\begin{equation*}
			\begin{aligned}
			\varUpsilon _2\lesssim&\int_{A _4}{\left| \frac{f_1\left( \xi _1,\tau _1 \right)}{\left< L_1 \right> ^{\frac{1}{2}}}\frac{\chi _{\left| \xi _2 \right|\le m}f_2\left( \xi _2,\tau _2 \right)}{\left< L_2 \right> ^{\frac{1}{2}}}\frac{\left| \xi _3 \right|^{\frac{1}{2}}\chi _{\left| \xi _3 \right|>4m}f_3\left( \xi _3,\tau _3 \right)}{\left< L_3 \right> ^{\frac{1}{2}}}\frac{\chi _{\left| \xi _4 \right|\ge 2m}f_4\left( \xi _4,\tau _4 \right)}{\left| \xi _4 \right|^{\frac{1}{4}}\left< L_4 \right> ^{\frac{7}{12}}} \right|}\\
			\lesssim&\lVert F_{\frac{1}{2}}^{(1)} \rVert _{L_{x}^{4}L_{t}^{4}}\lVert F_{\frac{1}{2}}^{(2)} \rVert _{L_{x}^{4}L_{t}^{4}}\lVert D_{x}^{\frac{1}{2}}P^{4m}F_{\frac{1}{2}}^{(3)} \rVert _{L_{x}^{4}L_{t}^{2}}\lVert D_{x}^{-\frac{1}{4}}P^{2m}F_{\frac{7}{12}}^{(4)} \rVert _{L_{x}^{4}L_{t}^{\infty}}\lesssim\prod_{j=1}^4{\lVert f_j \rVert _{L_{\xi}^{2}L_{\tau}^{2}}}.
			\end{aligned}
			\end{equation*}
			\item \textbf{Case C.}\quad$\left| \xi _3 \right|>4m,\ \left| \xi _2 \right|>m,\ \left| \xi _1 \right|\le 1,$\\
			It follows that $\mathcal{J}\lesssim \frac{\left< \xi _2 \right> ^p\left< \xi _3 \right> ^p}{\left< \xi _4 \right> ^2}\lesssim \frac{\left| \xi _3 \right|}{\left< \xi _4 \right> ^2}\sim\frac{\left| \xi _3 \right|}{\left< L_4 \right> ^{\frac{2}{3}}}$, $\mathcal{K}\lesssim\frac{1}{\left< L_1 \right> ^\sigma\left< L_2 \right> ^{\frac{1}{2}}\left< L_3 \right> ^{\frac{1}{2}}}$. 
			\subitem \textbf{Case C.(a). }\quad$\left< L_1 \right> =MAX$, then it follows that, $$\left< L_1 \right> ^\sigma\left< L_2 \right> ^{\frac{1}{2}}\left< L_3 \right> ^{\frac{1}{2}}\ge \left< L_1 \right> ^{\frac{\sigma}{2}+\frac{1}{4}}\left< L_2 \right> ^{\frac{1}{2}}\left< L_3 \right> ^{\frac{\sigma}{2}+\frac{1}{4}}.$$ Hence, the integral in this domain is bounded by:
			\begin{equation*}
			\begin{aligned}
			\varUpsilon _3\lesssim&\int_{A _4}{\left| \frac{\chi _{\left| \xi _1 \right|\le 1}f_1\left( \xi _1,\tau _1 \right)}{\left< L_1 \right> ^{\frac{\sigma}{2}+\frac{1}{4}}}\frac{f_2\left( \xi _2,\tau _2 \right)}{\left< L_2 \right> ^{\frac{1}{2}}}\frac{\left| \xi _3 \right|\chi _{\left| \xi _3 \right|>4m}f_3\left( \xi _3,\tau _3 \right)}{\left< L_3 \right> ^{\frac{\sigma}{2}+\frac{1}{4}}}\frac{f_4\left( \xi _4,\tau _4 \right)}{\left< L_4 \right> ^{\frac{2}{3}}} \right|}\\
			\lesssim&\lVert P_{1}F_{\frac{\sigma}{2}+\frac{1}{4}}^{(1)} \rVert _{L_{x}^{2}L_{t}^{\infty}}\lVert F_{\frac{1}{2}}^{(2)} \rVert _{L_{x}^{4}L_{t}^{4}}\lVert D_xP^{4m}F_{\frac{\sigma}{2}+\frac{1}{4}}^{(3)} \rVert _{L_{x}^{\infty}L_{t}^{2}}\lVert F_{\frac{2}{3}}^{(4)} \rVert _{L_{x}^{4}L_{t}^{4}}\lesssim\prod_{j=1}^4{\lVert f_j \rVert _{L_{\xi}^{2}L_{\tau}^{2}}}.
			\end{aligned}
			\end{equation*}
			\subitem \textbf{Case C.(b). }\quad$\left< L_2 \right> =MAX $, then it follows that, $$\left< L_1 \right> ^\sigma\left< L_2 \right> ^{\frac{1}{2}}\left< L_3 \right> ^{\frac{1}{2}}\ge \left< L_1 \right> ^\sigma\left< L_2 \right> ^{\frac{5}{12}}\left< L_3 \right> ^{\frac{7}{12}}.$$
			Hence, the integral in this domain is bounded by:
			\begin{equation*}
			\begin{aligned}
			\varUpsilon _3\lesssim&\int_{A _4}{\left| \frac{\chi _{\left| \xi _1 \right|\le 1}f_1\left( \xi _1,\tau _1 \right)}{\left< L_1 \right> ^{\sigma}}\frac{f_2\left( \xi _2,\tau _2 \right)}{\left< L_2 \right> ^{\frac{5}{6}}}\frac{\left| \xi _3 \right|\chi _{\left| \xi _3 \right|>4m}f_3\left( \xi _3,\tau _3 \right)}{\left< L_3 \right> ^{\frac{7}{6}}}\frac{f_4\left( \xi _4,\tau _4 \right)}{\left< L_4 \right> ^{\frac{2}{3}}} \right|}\\
			\lesssim&\lVert P_1F_{\sigma}^{(1)} \rVert _{L_{x}^{2}L_{t}^{\infty}}\lVert F_{\frac{5}{12}}^{(2)} \rVert _{L_{x}^{4}L_{t}^{4}}\lVert D_xP^{4m}F_{\frac{7}{12}}^{(3)} \rVert _{L_{x}^{\infty}L_{t}^{2}}\lVert F_{\frac{2}{3}}^{(4)} \rVert _{L_{x}^{4}L_{t}^{4}}\lesssim\prod_{j=1}^4{\lVert f_j \rVert _{L_{\xi}^{2}L_{\tau}^{2}}}.
			\end{aligned}
			\end{equation*}
			\subitem \textbf{Case C.(c). }\quad$\left< L_3 \right> =MAX $, we split our discussion in following subcases.
			\subsubitem \textbf{Case C.(c).1.}\quad$\left| \xi _4 \right|\gtrsim\left| \xi _3 \right|$. It follows that
			$\mathcal{J}\lesssim\frac{\left< \xi _2 \right> ^p\left< \xi _3 \right> ^p}{\left< \xi _4 \right> ^2}\lesssim1 .$
			Hence, the integral in this domain is bounded by $\prod_{j=1}^4{\lVert f_j \rVert _{L_{\xi}^{2}L_{\tau}^{2}}}$ using the argument similar to  \textbf{Case A}.
			\subsubitem \textbf{Case C.(c).2.}\quad$\left| \xi _4 \right|\ll\left| \xi _3 \right|$. Similar to \textbf{Case C.(c).2} in the proof of \eqref{1.10}, in this case, one has $\left| \xi _3 \right|\sim\left| \xi _2 \right|$, $\mathcal{J}\lesssim \frac{\left| \xi _3 \right|}{\left< L_4 \right> ^{\frac{2}{3}}}\sim \frac{\left| \xi _2 \right|}{\left< L_4 \right> ^{\frac{2}{3}}}$ and $\left< L_1 \right> ^\sigma\left< L_2 \right> ^{\frac{1}{2}}\left< L_3 \right> ^{\frac{1}{2}}\ge \left< L_1 \right> ^\sigma\left< L_2 \right> ^{\frac{7}{12}}\left< L_3 \right> ^{\frac{5}{12}}$. Hence, the integral in this domain is bounded by:
			\begin{equation*}
			\begin{aligned}
			\varUpsilon _3\lesssim&\int_{A _4}{\left| \frac{\chi _{\left| \xi _1 \right|\le 1}f_1\left( \xi _1,\tau _1 \right)}{\left< L_1 \right> ^\sigma}\frac{\chi _{\left| \xi _2 \right|>m}\left| \xi _2 \right|f_2\left( \xi _2,\tau _2 \right)}{\left< L_2 \right> ^{\frac{7}{12}}}\frac{f_3\left( \xi _3,\tau _3 \right)}{\left< L_3 \right> ^{\frac{5}{12}}}\frac{f_4\left( \xi _4,\tau _4 \right)}{\left< L_4 \right> ^{\frac{2}{3}}} \right|}\\
			\lesssim&\lVert P_1F_{\sigma}^{(1)} \rVert _{L_{x}^{2}L_{t}^{\infty}}\lVert D_xP^{m}F_{\frac{7}{12}}^{(2)} \rVert _{L_{x}^{\infty}L_{t}^{2}}\lVert F_{\frac{5}{12}}^{(3)} \rVert _{L_{x}^{4}L_{t}^{4}}\lVert F_{\frac{2}{3}}^{(4)} \rVert _{L_{x}^{4}L_{t}^{4}}\lesssim\prod_{j=1}^4{\lVert f_j \rVert _{L_{\xi}^{2}L_{\tau}^{2}}}.
			\end{aligned}
			\end{equation*}
			\subitem \textbf{Case C.(d). }\quad$\left< L_4 \right> =MAX $, then it follows that, $$\left< L_1 \right> ^\sigma\left< L_2 \right> ^{\frac{1}{2}}\left< L_3 \right> ^{\frac{1}{2}}\left< L_4 \right> ^{\frac{2}{3}}\ge \left< L_1 \right> ^\sigma\left< L_2 \right> ^{\frac{1}{2}}\left< L_3 \right> ^{\frac{2}{3}}\left< L_4 \right> ^{\frac{1}{2}}.$$
			Hence, the integral in this domain is bounded by:
			\begin{equation*}
			\begin{aligned}
			\varUpsilon _3\lesssim&\int_{A_4}{\left| \frac{\chi _{\left| \xi _1 \right|\le 1}f_1\left( \xi _1,\tau _1 \right)}{\left< L_1 \right> ^\sigma}\frac{f_2\left( \xi _2,\tau _2 \right)}{\left< L_2 \right> ^{\frac{1}{2}}}\frac{\left| \xi _3 \right|\chi _{\left| \xi _3 \right|>4m}f_3\left( \xi _3,\tau _3 \right)}{\left< L_3 \right> ^{\frac{2}{3}}}\frac{f_4\left( \xi _4,\tau _4 \right)}{\left< L_4 \right> ^{\frac{1}{2}}} \right|}\\
			\lesssim&\lVert P_{1}F_{\sigma}^{(1)} \rVert _{L_{x}^{2}L_{t}^{\infty}}\lVert F_{\frac{1}{2}}^{(2)} \rVert _{L_{x}^{4}L_{t}^{4}}\lVert D_xP^{4m}F_{\frac{2}{3}}^{(3)} \rVert _{L_{x}^{\infty}L_{t}^{2}}\lVert F_{\frac{1}{2}}^{(4)} \rVert _{L_{x}^{4}L_{t}^{4}}\lesssim\prod_{j=1}^4{\lVert f_j \rVert _{L_{\xi}^{2}L_{\tau}^{2}}}.
			\end{aligned}
			\end{equation*}
			\item \textbf{Case D.}\quad$\left| \xi _3 \right|>4m,\ \left| \xi _2 \right|>m,\ \left| \xi _1 \right|> 1$, $\left| \xi _4 \right|\sim\left| \xi _3 \right|.$\\
			It follows that $\mathcal{J}\lesssim\frac{\left< \xi _3 \right> ^p\left< \xi _3 \right> ^p\left< \xi _3 \right> ^p}{\left< \xi _3 \right> ^2}\lesssim1$ since $\left| \xi _1 \right|\le \left| \xi _2 \right|\le \left| \xi _3 \right|$ and $p<\frac{1}{2}$. Then, by the arguement similar to \textbf{Case A}, we can obtain that $\varUpsilon _4\lesssim\prod_{j=1}^4{\lVert f_j \rVert _{L_{\xi}^{2}L_{\tau}^{2}}}.$
			\item \textbf{Case E.}\quad$\left| \xi _3 \right|>4m,\ \left| \xi _2 \right|>m,\ \left| \xi _1 \right|> 1$, $\left| \xi _1 \right|\lesssim\left| \xi _4 \right|\ll\left| \xi _2 \right|\sim\left| \xi _3 \right|.$ 
			\subitem \textbf{Case E.(a).}\quad$\left< L_1 \right> =MAX$, then for some $\epsilon_j>0$ small enough, $j=1,2,3$, one has,  
			$$\mathcal{J}\lesssim \frac{\left< \xi _3 \right> ^{2p}}{\left< \xi _4 \right> ^{2-p}}\lesssim \frac{\left| \xi _3 \right|}{\left< \xi _4 \right> ^{\frac{3}{4}+\epsilon _1}\left< \xi _4 \right> ^{\frac{5}{4}-p-\epsilon _1}}\sim \frac{\left| \xi _3 \right|}{\left< \xi _4 \right> ^{\frac{3}{4}+\epsilon _1}\left< L_4 \right> ^{\frac{5}{12}-\frac{p+\epsilon _1}{3}}},$$ $$\left< L_1 \right> ^{\frac{1}{2}}\left< L_2 \right> ^{\frac{1}{2}}\left< L_3 \right> ^{\frac{1}{2}}\ge \left< L_1 \right> ^{\frac{2}{9}+\epsilon _2}\left< L_2 \right> ^{\frac{1}{2}}\left< L_3 \right> ^{\frac{1}{2}+\epsilon _3}\left< L_4 \right> ^{\frac{5}{18}-\epsilon _2-\epsilon _3},$$
			and
			$$\mathcal{J}\mathcal{K}\lesssim \frac{\left| \xi _3 \right|}{\left< \xi _4 \right> ^{\frac{3}{4}+\epsilon _1}\left< L_1 \right> ^{\frac{2}{9}+\epsilon _2}\left< L_2 \right> ^{\frac{1}{2}}\left< L_3 \right> ^{\frac{1}{2}+\epsilon _3}\left< L_4 \right> ^{\frac{25}{36}-\frac{p+\epsilon _1}{3}-\epsilon _2-\epsilon _3}}.$$
			We intend to use $L^2_xL^{\infty}_{t}$ norm to control $J_{x}^{-\left( \frac{3}{4}+\epsilon _1 \right)}F_{\rho}^{(4)}$ by \eqref{2.13}, where $\rho=\frac{25}{36}-\frac{p+\epsilon _1}{3}-\epsilon _2-\epsilon _3$, thus, it forces that $\rho>\frac{1}{2}$, hence, one may require:
			$$\frac{7}{36}-\frac{p}{3}>\frac{1}{36}>\frac{\epsilon _1}{3}+\epsilon _2+\epsilon _3,$$
			for $p<\frac{1}{2}$. Particularly, we can choose $\epsilon_1=\frac{1}{72}$, $\epsilon_2=\epsilon_3=\frac{1}{144}$, therefore, 
			$$\mathcal{J}\mathcal{K}\lesssim \frac{\left| \xi _3 \right|}{\left< \xi _4 \right>^{\frac{55}{72}}\left< L_1 \right> ^{\frac{33}{144}}\left< L_2 \right> ^{\frac{1}{2}}\left< L_3 \right> ^{\frac{73}{144}}\left< L_4 \right> ^{\frac{73}{108}-\frac{p}{3}}}.$$ 
			Finally, the integral in this domain is bounded by:
			\begin{equation*}
			\begin{aligned}
			\varUpsilon _5\lesssim&\int_{A_4}{\left| \frac{f_1\left( \xi _1,\tau _1 \right)}{\left< L_1 \right> ^{\frac{33}{144}}}\frac{f_2\left( \xi _2,\tau _2 \right)}{\left< L_2 \right> ^{\frac{1}{2}}}\frac{\left| \xi _3 \right|\chi _{\left| \xi _3 \right|\ge 4m}f_3\left( \xi _3,\tau _3 \right)}{\left< L_3 \right> ^{\frac{73}{144}}}\frac{f_4\left( \xi _4,\tau _4 \right)}{\left< \xi _4 \right>^\frac{55}{72}\left< L_4 \right> ^{\frac{73}{108}-\frac{p}{3}}} \right|}\\
			\lesssim&\lVert F_{\frac{33}{144}}^{(1)} \rVert _{L_{x}^{3}L_{t}^{3}}\lVert F_{\frac{1}{2}}^{(2)} \rVert _{L_{x}^{6}L_{t}^{6}}\lVert D_xP^{4m}F_{\frac{73}{144}}^{(3)} \rVert _{L_{x}^{\infty}L_{t}^{2}}\lVert J_{x}^{-\frac{55}{72}}F_{\frac{73}{108}-\frac{p}{3}}^{(4)} \rVert _{L_{x}^{2}L_{t}^{\infty}}\lesssim\prod_{j=1}^4{\lVert f_j \rVert _{L_{\xi}^{2}L_{\tau}^{2}}},
			\end{aligned}
			\end{equation*}
			since $\frac{73}{108}-\frac{p}{3}>\frac{55}{108}>\frac{1}{2}$.
			\subitem \textbf{Case E.(b).}\quad$\left< L_2 \right> =MAX $, similar to 
			\textbf{Case E.(a)},  it follows that, $$\mathcal{J}\lesssim\frac{\left| \xi _3 \right|}{\left< \xi _4 \right>^{\frac{55}{72}}\left< L_4 \right> ^{\frac{1}{3}\left( \frac{89}{72}-p \right)}},\quad\left< L_1 \right> ^{\frac{1}{2}}\left< L_2 \right> ^{\frac{1}{2}}\left< L_3 \right> ^{\frac{1}{2}}\ge \left< L_1 \right> ^{\frac{1}{2}}\left< L_2 \right> ^{\frac{33}{144}}\left< L_3 \right> ^{\frac{73}{144}}\left< L_4 \right> ^{\frac{19}{72}}$$
			and 
			$$\mathcal{J}\mathcal{K}\lesssim \frac{\left| \xi _3 \right|}{\left< \xi _4 \right>^{\frac{55}{72}}\left< L_1 \right> ^{\frac{1}{2}}\left< L_2 \right> ^{\frac{33}{144}}\left< L_3 \right> ^{\frac{73}{144}}\left< L_4 \right> ^{\frac{73}{108}-\frac{p}{3}}}.$$
			Hence, the integral in this domain is bounded by:
			\begin{equation*}
			\varUpsilon _5\lesssim\lVert F_{\frac{1}{2}}^{(1)} \rVert _{L_{x}^{6}L_{t}^{6}}\lVert F_{\frac{33}{144}}^{(2)} \rVert _{L_{x}^{3}L_{t}^{(3)}}\lVert D_xP^{4m}F_{\frac{73}{144}}^{(3)} \rVert _{L_{x}^{\infty}L_{t}^{2}}\lVert J_{x}^{-\frac{55}{72}}F_{\frac{73}{108}-\frac{p}{3}}^{(4)} \rVert _{L_{x}^{2}L_{t}^{\infty}}\lesssim \prod_{j=1}^4{\lVert f_j \rVert _{L_{\xi}^{2}L_{\tau}^{2}}}.
			\end{equation*}
			\subitem \textbf{Case E.(c).}\quad$\left< L_3 \right> =MAX $, then it follows that, $$\mathcal{J}\lesssim \frac{\left| \xi _3 \right|}{\left< \xi _4 \right>^{\frac{55}{72}}\left< L_4 \right> ^{\frac{1}{3}\left( \frac{89}{72}-p \right)}}\sim\frac{\left| \xi _2 \right|}{\left< \xi _4 \right>^{\frac{55}{72}}\left< L_4 \right> ^{\frac{1}{3}\left( \frac{89}{72}-p \right)}},$$$$\left< L_1 \right> ^{\frac{1}{2}}\left< L_2 \right> ^{\frac{1}{2}}\left< L_3 \right> ^{\frac{1}{2}}\ge \left< L_1 \right> ^{\frac{1}{2}}\left< L_2 \right> ^{\frac{73}{144}}\left< L_3 \right> ^{\frac{33}{144}}\left< L_4 \right> ^{\frac{19}{72}},$$
			and 
			$$\mathcal{J}\mathcal{K}\lesssim \frac{\left| \xi _2 \right|}{\left< \xi _4 \right>^{\frac{55}{72}}\left< L_1 \right> ^{\frac{1}{2}}\left< L_2 \right> ^{\frac{73}{144}}\left< L_3 \right> ^{\frac{33}{144}}\left< L_4 \right> ^{\frac{73}{108}-\frac{p}{3}}}.$$
			Hence, the integral in this domain is bounded by:
			\begin{equation*}
			\varUpsilon _5\lesssim\lVert F_{\frac{1}{2}}^{(1)} \rVert _{L_{x}^{6}L_{t}^{6}}\lVert D_xP^{m}F_{\frac{73}{144}}^{(2)} \rVert _{L_{x}^{\infty}L_{t}^{2}}\lVert F_{\frac{33}{144}}^{(3)} \rVert _{L_{x}^{3}L_{t}^{3}}\lVert J_{x}^{-\frac{55}{72}}F_{\frac{73}{108}-\frac{p}{3}}^{(4)} \rVert _{L_{x}^{2}L_{t}^{\infty}}\lesssim \prod_{j=1}^4{\lVert f_j \rVert _{L_{\xi}^{2}L_{\tau}^{2}}}.
			\end{equation*}
			\subitem \textbf{Case E.(d).}\quad$\left< L_4 \right> =MAX $, then it follows that, $$\mathcal{J}\lesssim \frac{\left| \xi _3 \right|}{\left< \xi _4 \right>^{\frac{55}{72}}\left< L_4 \right> ^{\frac{1}{3}\left( \frac{89}{72}-p \right)}}\lesssim\frac{\left| \xi _3 \right|}{\left< \xi _1 \right>^{\frac{55}{72}}\left< L_1 \right> ^{\frac{1}{144}}\left< L_3 \right> ^{\frac{1}{144}}\left< L_4 \right> ^{\frac{43}{108}-\frac{p}{3}}}$$
			and 
			$$\mathcal{J}\mathcal{K}\lesssim\frac{\left| \xi _3 \right|}{\left< \xi_1 \right>^{\frac{55}{72}}\left< L_1 \right> ^{\frac{73}{144}}\left< L_2 \right> ^{\frac{1}{2}}\left< L_3 \right> ^{\frac{73}{144}}\left< L_4 \right> ^{\frac{43}{108}-\frac{p}{3}}}.$$
			Hence, the integral in this domain is bounded by:
			\begin{equation*}
			\varUpsilon _5\lesssim\lVert J_{x}^{-\frac{55}{72}}F_{\frac{73}{144}}^{(1)} \rVert _{L_{x}^{2}L_{t}^{\infty}}\lVert F_{\frac{1}{2}}^{(2)} \rVert _{L_{x}^{6}L_{t}^{6}}\lVert D_xP^{4m}F_{\frac{73}{144}}^{(3)} \rVert _{L_{x}^{\infty}L_{t}^{2}}\lVert F_{\frac{73}{108}-\frac{p}{3}}^{(4)} \rVert _{L_{x}^{3}L_{t}^{3}}\lesssim \prod_{j=1}^4{\lVert f_j \rVert _{L_{\xi}^{2}L_{\tau}^{2}}},
			\end{equation*}
			since $\frac{43}{108}-\frac{p}{3}>\frac{25}{108}>\frac{2}{9}.$
			\item \textbf{Case F.}\quad$\left| \xi _3 \right|>4m,\ \left| \xi _2 \right|>m,\ \left| \xi _1 \right|> 1$, $\left| \xi _4 \right|\ll\left| \xi _1 \right|$, $\left| \xi _2 \right|\sim\left| \xi _3 \right|,$\\
			By \textbf{Lemma} \ref{lem5.1}, in this whole domain, we can give a uniform bound of the resonance function $H$:
			\begin{equation*}
			\left| H \right|\gtrsim\left| \xi _1\xi _2\xi _3 \right|.
			\end{equation*}
			Then by the argument similar as before, we can conclude that: 
			$$\varUpsilon _6\lesssim\prod_{j=1}^4{\lVert f_j \rVert _{L_{\xi}^{2}L_{\tau}^{2}}}.$$
		\end{itemize}
		The proof is now complete.\end{proof}
	\section{Studies on the IVP}
	\subsection{Related linear estimates for IVP}
	Following estimates are essential to establish the local wellposedness and can be found in \cite{esfahani2012local,esfahani2014bilinear}.
	\begin{lem}\label{lem3.1}
		For $s,b\in \mathbb{R}$, then
		$$\lVert \eta \left( t \right) \left[ W_{\mathbb{R}}\left( \varphi ,\psi \right) \right]\left(x,t\right) \rVert _{X^{s,b}}\lesssim\lVert \varphi \rVert _{H^s\left( \mathbb{R} \right)}+\lVert \psi \rVert _{H^{s-1}\left( \mathbb{R} \right)}.$$
	\end{lem} 
	\begin{lem}\label{lem3.2}
		For $s\in \mathbb{R}$, $-\frac{1}{2}<b^{\prime}\le 0\le b\le b^{\prime}+1$ and $0<T<1$, then
		$$\Big\lVert \eta \left( \frac{t}{T} \right) \int_0^t{\left[ W_{\mathbb{R}}\left( 0,f \right) \right] \left( x,t-\tau \right) d\tau} \Big\rVert _{X^{s,b}}\le T^{1+b^{\prime}-b}\Big\lVert \left( \frac{\hat{f}\left( \xi ,\tau \right)}{\phi \left( \xi \right)} \right) ^{\lor} \Big\rVert _{X^{s,b^{\prime}}}.$$
	\end{lem}
	By \textbf{Lemma} \ref{lem3.1}, \ref{3.2}, \textbf{Theorem} \ref{thm1.4} and contraction mapping principal, we can obtain \textbf{Theorem} \ref{thm1.6}.
	
	\subsection{Proof of Theorem \ref{thm1.4}, multilinear estimates in $X^{s,b}$ with $b>\frac{1}{2}$}
	\begin{proof}
		Set $p=-s>0$, based on duality and Plancherel’s equality, it suffices to prove:
		\begin{equation}\label{4.3}
		\left| \int_{A_4}{\frac{\left< \xi _1 \right> ^p\left< \xi _2 \right> ^p\left< \xi _3 \right> ^p\prod_{j=1}^4{f_j\left( \xi _j,\tau _j \right)}}{\left< \xi _4 \right> ^{p+1}\left< L_1 \right> ^b\left< L_2 \right> ^b\left< L_3 \right> ^b\left< L_4 \right> ^{1-b}}} \right|\lesssim\prod_{j=1}^4{\lVert f_j \rVert _{L_{\xi}^{2}L_{\tau}^{2}}},
		\end{equation}
		where \begin{equation*}
		\begin{aligned}
		&f_1=\left< \xi \right> ^s\left< \left| \tau \right|-\left| \xi \right|^3+\frac{1}{2}\left| \xi \right| \right> ^b\widehat{u}\left( \xi ,\tau \right),\quad f_2=\left< \xi \right> ^s\left< \left| \tau \right|-\left| \xi \right|^3+\frac{1}{2}\left| \xi \right| \right> ^b\widehat{v}\left( \xi ,\tau \right),\\
		&f_3=\left< \xi \right> ^s\left< \left| \tau \right|-\left| \xi \right|^3+\frac{1}{2}\left| \xi \right| \right> ^b\widehat{w}\left( \xi ,\tau \right),\quad f_j\in L^2\left( \mathbb{R}^2 \right),j=1,2,3,4,
		\end{aligned}
		\end{equation*}and 
		\begin{equation*}
		L_j=\tau _j-sgn \left( \tau _j \right) \left( \left| \xi _j \right|^3-\frac{1}{2}\left| \xi _j \right| \right), j=1,2,3,4.
		\end{equation*}
		Without loss of generality, we assume $\frac{1}{3}<p<\frac{1}{2}$ and take $b_0=1-p<\frac{2}{3}$. In addition, we also assume $\left| \xi _1 \right|\le \left| \xi _2 \right|\le \left| \xi _3 \right|$ which impies that $\left| \xi _4 \right|\lesssim\left| \xi _3 \right|$. In order to establish the estimate \eqref{4.3}, we split the hyperplane into the following pieces:\begin{itemize}
			\item \textbf{Case \uppercase\expandafter{A}.}\quad$\left| \xi _3 \right|\le m,$
			\item \textbf{Case \uppercase\expandafter{B}.}\quad$\left| \xi _3 \right|> m$, $\left| \xi _2 \right|\le m,$ 
			\item \textbf{Case \uppercase\expandafter{C}.}\quad$\left| \xi _3 \right|,\left| \xi _2 \right|> m$, $\left| \xi _1 \right|\le m,$
			\item \textbf{Case \uppercase\expandafter{D}.}\quad$\left| \xi _3 \right|,\left| \xi _2 \right|,\left| \xi _1 \right|> m$, $\left| \xi _4 \right|\sim\left| \xi _3 \right|,$
			\item \textbf{Case \uppercase\expandafter{E}.}\quad$\left| \xi _3 \right|,\left| \xi _2 \right|,\left| \xi _1 \right|>m$, $\left| \xi _1 \right|\lesssim\left| \xi _4 \right|\ll\left| \xi _2 \right|\sim\left| \xi _3 \right|,$
			\item \textbf{Case \uppercase\expandafter{F}.}\quad$\left| \xi _3 \right|,\left| \xi _2 \right|,\left| \xi _1 \right|>m$, $\left| \xi _4 \right|\ll\left| \xi _1 \right|$, $\left| \xi _2 \right|\sim\left| \xi _3 \right|,$
		\end{itemize}
		For simplicity in natations, we write,
		\begin{equation*}
		J:=\frac{\left< \xi _1 \right> ^p\left< \xi _2 \right> ^p\left< \xi _3 \right> ^p}{\left< \xi _4 \right> ^{p+1}}
		\end{equation*}
		and denote $\varPsi _j$ as the integral over the domains of \textbf{Case A} through \textbf{Case F}, accordingly. Nextm, we will prove that $\varPsi _j$ is bounded by $\prod_{j=1}^4{\lVert f_j \rVert _{L_{\xi}^{2}L_{\tau}^{2}}}$.
		\begin{itemize}
			\item \textbf{Case \uppercase\expandafter{A}.} \quad $\left| \xi _3 \right|\le m.$\\
			It follows that $\left| \xi _1 \right|$, $\left| \xi _2 \right|\lesssim1$, thus, $J\lesssim1$. Hence, the integral in this domain is bounded by:
			\begin{equation*}
			\begin{aligned}
			\varPsi _1\lesssim&\int{\left| \frac{f_1\left( \xi _1,\tau _1 \right)}{\left< L_1 \right> ^b}\frac{f_2\left( \xi _2,\tau _2 \right)}{\left< L_2 \right> ^b}\frac{f_3\left( \xi _3,\tau _3 \right)}{\left< L_3 \right> ^b}\frac{f_4\left( \xi _4,\tau _4 \right)}{\left< L_4 \right> ^{1-b}} \right|}\\
			=&\lesssim\lVert F_{1-b}^{(4)} \rVert _{L_{x}^{2}L_{t}^{2}}\prod_{j=1}^3{\lVert F_{b}^{(j)} \rVert _{L_{x}^{6}L_{t}^{6}}}\lesssim\prod_{j=1}^4{\lVert f_j \rVert _{L_{\xi}^{2}L_{\tau}^{2}}},
			\end{aligned}
			\end{equation*}
			\item \textbf{Case \uppercase\expandafter{B}.}\quad$\left| \xi _3 \right|> m$, $\left| \xi _2 \right|\le m.$\\
			It follows that $J\lesssim\left< \xi _3 \right> ^p\lesssim\left| \xi _3 \right|$ for $p<1$. Hence, the integral in this domain is bounded by:
			\begin{equation*}
			\begin{aligned}
			\varPsi _2\lesssim&\int{\left| \frac{f_1\left( \xi _1,\tau _1 \right)}{\left< L_1 \right> ^b}\frac{\chi _{\left| \xi _2 \right|\le m}f_2\left( \xi _2,\tau _2 \right)}{\left< L_2 \right> ^b}\frac{\left| \xi _3 \right|\chi _{\left| \xi _3 \right|\ge m}f_3\left( \xi _3,\tau _3 \right)}{\left< L_3 \right> ^b}\frac{f_4\left( \xi _4,\tau _4 \right)}{\left< L_4 \right> ^{1-b}} \right|}\\
			\lesssim&\lVert F_{b}^{(1)} \rVert _{L_{x}^{4}L_{t}^{4}}\lVert P_mF_{b}^{(2)} \rVert _{L_{x}^{2}L_{t}^{\infty}}\lVert D_xP^mF_{b}^{(3)} \rVert _{L_{x}^{\infty}L_{t}^{2}}\lVert F_{1-b}^{(4)} \rVert _{L_{x}^{4}L_{t}^{4}}\lesssim\prod_{j=1}^4{\lVert f_j \rVert _{L_{\xi}^{2}L_{\tau}^{2}}},
			\end{aligned}
			\end{equation*}
			\item \textbf{Case \uppercase\expandafter{C}.}\quad$\left| \xi _3 \right|,\left| \xi _2 \right|> m$, $\left| \xi _1 \right|\le m.$\\
			It follows that $J_1\lesssim \left< \xi _2 \right> ^p\left< \xi _3 \right> ^p\lesssim \left| \xi _3 \right|$ for $p<\frac{1}{2}$. Then, one has,
			\begin{equation*}
			\begin{aligned}
			\varPsi _3\lesssim&\int{\left| \frac{\chi _{\left| \xi _1 \right|\le m}f_1\left( \xi _1,\tau _1 \right)}{\left< L_1 \right> ^b}\frac{f_2\left( \xi _2,\tau _2 \right)}{\left< L_2 \right> ^b}\frac{\left| \xi _3 \right|\chi _{\left| \xi _3 \right|\ge m}f_3\left( \xi _3,\tau _3 \right)}{\left< L_3 \right> ^b}\frac{f_4\left( \xi _4,\tau _4 \right)}{\left< L_4 \right> ^{1-b}} \right|}\\
			\lesssim&\lVert P_mF_{b}^{(1)} \rVert _{L_{x}^{2}L_{t}^{\infty}}\lVert F_{b}^{(2)} \rVert _{L_{x}^{4}L_{t}^{4}}\lVert D_xP^mF_{b}^{(3)} \rVert _{L_{x}^{\infty}L_{t}^{2}}\lVert F_{1-b}^{(4)} \rVert _{L_{x}^{4}L_{t}^{4}}\lesssim\prod_{j=1}^4{\lVert f_j \rVert _{L_{\xi}^{2}L_{\tau}^{2}}},
			\end{aligned}
			\end{equation*}
			by \textbf{Lemma} \ref{lem1.7}, \ref{lem2.8}, \ref{lem2.9}, \ref{2.10}.
			\item \textbf{Case \uppercase\expandafter{D}.}\quad$\left| \xi _3 \right|,\left| \xi _2 \right|,\left| \xi _1 \right|> m$, $\left| \xi _4 \right|\sim\left| \xi _3 \right|.$\\
			It follows that $J\lesssim\frac{\left< \xi _3 \right> ^p\left< \xi _3 \right> ^p\left< \xi _3 \right> ^p}{\left< \xi _3 \right> ^{p+1}}\lesssim1$ since $\left| \xi _1 \right|\le \left| \xi _2 \right|\le \left| \xi _3 \right|$ and $p<\frac{1}{2}$. Then, by argument similar to \textbf{Case \uppercase\expandafter{A}}, we can obtain that $\varPsi _4\lesssim\prod_{j=1}^4{\lVert f_j \rVert _{L_{\xi}^{2}L_{\tau}^{2}}}.$
			\item \textbf{Case \uppercase\expandafter{E}.}\quad$\left| \xi _3 \right|,\left| \xi _2 \right|,\left| \xi _1 \right|> m$, $\left| \xi _1 \right|\lesssim\left| \xi _4 \right|\ll\left| \xi _2 \right|\sim\left| \xi _3 \right|.$\\
			It follows that $J\lesssim\frac{\left< \xi _2 \right> ^p\left< \xi _3 \right> ^p}{\left< \xi _1 \right>}\lesssim\frac{\left| \xi _3 \right|}{\left< \xi _1 \right>}$. Thus, one has,
			\begin{equation*}
			\begin{aligned}
			\varPsi _5\lesssim&\int{\left| \frac{\chi _{\left| \xi _1 \right|\ge m}f_1\left( \xi _1,\tau _1 \right)}{\left< \xi _1 \right> \left< L_1 \right> ^b}\frac{f_2\left( \xi _2,\tau _2 \right)}{\left< L_2 \right> ^b}\frac{\left| \xi _3 \right|\chi _{\left| \xi _3 \right|\ge m}f_3\left( \xi _3,\tau _3 \right)}{\left< L_3 \right> ^b}\frac{f_4\left( \xi _4,\tau _4 \right)}{\left< L_4 \right> ^{1-b}} \right|}\\
			\lesssim&\lVert J_{x}^{-1}P^mF_{b}^{(1)} \rVert _{L_{x}^{2}L_{t}^{\infty}}\lVert F_{b}^{(2)} \rVert _{L_{x}^{4}L_{t}^{4}}\lVert D_xP^mF_{b}^{(3)} \rVert _{L_{x}^{\infty}L_{t}^{2}}\lVert F_{1-b}^{(4)} \rVert _{L_{x}^{4}L_{t}^{4}}\lesssim\prod_{j=1}^4{\lVert f_j \rVert _{L_{\xi}^{2}L_{\tau}^{2}}},
			\end{aligned}
			\end{equation*}
			by \textbf{Lemma} \ref{lem1.7}, \ref{lem2.9}, \ref{lem2.10}, since $\frac{1}{2}<b<\frac{2}{3}$.
			\item \textbf{Case \uppercase\expandafter{F}.}\quad$\left| \xi _3 \right|,\left| \xi _2 \right|,\left| \xi _2 \right|> m$, $\left| \xi _4 \right|\ll\left| \xi _1 \right|$, $\left| \xi _2 \right|\sim\left| \xi _3 \right|.$\\
			By \textbf{Lemma} \ref{lem5.1}, in this whole domain, we can give a uniform bound of the resonance function $H$:
			\begin{equation*}
			\left| H \right|\gtrsim\left| \xi _1\xi _2\xi _3 \right|.
			\end{equation*}
			Then, by the argument similar to \textbf{Case F} in the proof of \textbf{Theorem} \ref{thm1.3}, we can conclude that the integral in this domain is bounded by: 
			$$\varPsi _6\lesssim\prod_{j=1}^4{\lVert f_j \rVert _{L_{\xi}^{2}L_{\tau}^{2}}}.$$
		\end{itemize}
		The proof is now complete.\end{proof}
	\section{Appendix}\label{section5}
	For the ``hyperplane'' $A_k$, we introduce $L_j:=\tau _j-sgn \left( \tau _j \right) \left( \left| \xi _j \right|^3-\frac{1}{2}\left| \xi _j \right| \right)$, $j=1,2,...,k,$ and the resonance function,
	$$H\left( \xi_{1}, \xi_{2},..., \xi_{k} \right) =\sum_{j=1}^k{L_j},$$
	which is a key to handle the estimate in high frequencies part. To obtain low index multilinear estimates, we need to show uniform bounds for resonance function in the special domain.
	\begin{lem}\label{lem5.1}
		For $ \vec{\xi }\in\Omega$, where $\Omega:=$ $$ \left\{ \vec{\xi }=\left( \xi _1,\xi _2,\xi _3,\xi _4 \right) :\sum_{j=1}^4{\xi _j}=0,\left| \xi _1 \right|\le \left| \xi _2 \right|\le \left| \xi _3 \right|, \left| \xi _2 \right|\gg1, \left| \xi _3 \right|\gg 1, \left| \xi _1 \right|>1, \left| \xi _4 \right|\ll\left| \xi _1 \right|, \left| \xi _2 \right|\sim \left| \xi _3 \right|\right\},$$
		the relation $\left| H \right|\gtrsim \left| \xi _1\xi _2\xi _3 \right|$ holds.
	\end{lem}
	\begin{proof}
		In order to estimate the resonance function in this domain, we will examine the following seven potential cases regarding the signs of $\tau _j$:
		\begin{equation*}
		\begin{aligned}
		&A:=\left\{ \tau _1\ge 0,\tau _2\ge 0,\tau _3\ge 0,\tau _4\le 0 \right\}, \quad\quad B:=\left\{ \tau _1\ge 0,\tau _2\ge 0,\tau _3\le 0,\tau _4\ge 0 \right\}, \\
		&C:=\left\{ \tau _1\ge 0,\tau _2\ge 0,\tau _3\le 0,\tau _4\le 0 \right\}, \quad\quad
		D:=\left\{ \tau _1\ge 0,\tau _2\le 0,\tau _3\ge 0,\tau _4\ge 0 \right\}, \\
		&E:=\left\{ \tau _1\ge 0,\tau _2\le 0,\tau _3\ge 0,\tau _4\le 0 \right\}, \quad\quad 
		F:=\left\{ \tau _1\ge 0,\tau _2\le 0,\tau _3\le 0,\tau _4\ge 0 \right\}, \\ 
		&G:=\left\{ \tau _1\ge 0,\tau _2\le 0,\tau _3\le 0,\tau _4\le 0 \right\}. 
		\end{aligned}
		\end{equation*}
		Notice that $\vec{\tau }>0$ or $\vec{\tau }<0$ is impossible to happen due to $\sum_{j=1}^4{\tau _j}=0$, meanwhile, the remaining cases are equivalent to cases $A$, $B$, $C$, $D$, $E$, $F$ by considering the reflection transformation $\vec{\tau }=\left( \tau _1,\tau _2,\tau _3,\tau _4 \right) \rightarrow -\vec{\tau }$ and the fact that $\left| H\left( \vec{\xi } \right) \right|$ is invariant under reflection. Hence, it suffices to consider the above seven cases regarding the signs of $\tau _j$. In each cases, we also split the hyperplane $\left\{ \vec{\xi }\in \mathbb{R}^4:\sum_{j=1}^4{\xi _j}=0 \right\} $ as follows:
		\begin{equation*}
		\begin{aligned}
		&(a):=\left\{ \xi _1\ge 0,\xi _2\ge 0,\xi _3\ge 0,\xi _4\le 0 \right\},\quad (b):=\left\{ \xi _1\ge 0,\xi _2\ge 0,\xi _3\le 0,\xi _4\ge 0 \right\},\\
		&(c):=\left\{ \xi _1\ge 0,\xi _2\ge 0,\xi _3\le 0,\xi _4\le 0 \right\},\quad (d):=\left\{ \xi _1\ge 0,\xi _2\le 0,\xi _3\ge 0,\xi _4\ge 0 \right\},\\
		&(e):=\left\{ \xi _1\ge 0,\xi _2\le 0,\xi _3\ge 0,\xi _4\le 0 \right\},\quad (f):=\left\{ \xi _1\ge 0,\xi _2\le 0,\xi _3\le 0,\xi _4\ge 0 \right\},\\
		&(g):=\left\{ \xi _1\ge 0,\xi _2\le 0,\xi _3\le 0,\xi _4\le 0 \right\}.
		\end{aligned}	 
		\end{equation*}
		First, case $(a)$ will lead to $\left| \xi _4 \right|=\left| \xi _1 \right|+\left| \xi _2 \right|+\left| \xi _3 \right|$, which is contradicting with $\left| \xi _2 \right|\gg\left| \xi _4 \right|$. Next, case $(d)$ implies that $\left| \xi _2 \right|=\left| \xi _1 \right|+\left| \xi _3 \right|+\left| \xi _4 \right|$, this is contradicting with $\left| \xi _2 \right|\le\left| \xi _3 \right|$. Case $(e)$ will lead to $\left| \xi _1 \right|+\left| \xi _3 \right|=\left| \xi _2 \right|+\left| \xi _4 \right|$ which implies that $\left| \xi _4 \right|\ge \left| \xi _1 \right|$, this is also contradicting with $\left| \xi _4 \right|\ll\left| \xi _1 \right|$. Meanwhile, case $(f)$ implies that $\left| \xi _2 \right|+\left| \xi _3 \right|=\left| \xi _1 \right|+\left| \xi _4 \right|$, which is contradicting with $\left| \xi _2 \right|\gg\left| \xi _4 \right|$ and $\left| \xi _3 \right|\ge\left| \xi _1 \right|$. In addition, if case $(g)$ holds, then this will lead to $\left| \xi _1 \right|=\left| \xi _2 \right|+\left| \xi _3 \right|+\left| \xi _4 \right|$ which is also contradictory. Furthermore, for the same reason as before, the remaining cases are equivalent to cases $(a)$, $(b)$, $(c)$, $(d)$, $(e)$, $(f)$, $(g)$. Therefore, we only need to examine subcases $(b)$, $(c)$ in cases $A$, $B$, $C$, $D$, $E$, $F$, $G$. When $\vec{\xi }$ lies in $(b)$, then we can observe that $$\left| \xi _3 \right|=\left| \xi _1 \right|+\left| \xi _2 \right|+\left| \xi _4 \right|,\quad \left| \xi _3 \right|^3=\left| \xi _1 \right|^3+\left| \xi _2 \right|^3+\left| \xi _4 \right|^3+3\left( \left| \xi _1 \right|+\left| \xi _2 \right| \right) \left( \left| \xi _1 \right|+\left| \xi _4 \right| \right) \left( \left| \xi _2 \right|+\left| \xi _4 \right| \right).$$ When $\vec{\xi }$ lies in $(c)$, then we can observe that $$\left| \xi _3 \right|+\left| \xi _4 \right|=\left| \xi _1 \right|+\left| \xi _2 \right|,\quad\left| \xi _3 \right|^3=\left| \xi _1 \right|^3+\left| \xi _2 \right|^3-\left| \xi _4 \right|^3+3\left( \left| \xi _1 \right|+\left| \xi _2 \right| \right) \left( \left| \xi _1 \right|-\left| \xi _4 \right| \right) \left( \left| \xi _2 \right|-\left| \xi _4 \right| \right).$$
		The above relations are the key to estimate the resonance function. Next, we will deduce the inequalities in each cases. Since $\left| \xi _2 \right|, \left| \xi _3 \right|\gg 1, \left| \xi _1 \right|>1$, we can only focus on the monomials of highest order in $H\left( \vec{\xi } \right)$.
		\begin{itemize}
			\item \textbf{Case A.(b)} : $\tau _1\ge 0,\tau _2\ge 0,\tau _3\ge 0,\tau _4\le 0,\xi _1\ge 0,\xi _2\ge 0,\xi _3\le 0,\xi _4\ge 0$.
			\begin{equation*}
			\begin{aligned}
			\left| H \right|=&\left| -\xi _{1}^{3}+\frac{1}{2}\xi _1-\xi _{2}^{3}+\frac{1}{2}\xi _2+\xi _{3}^{3}-\frac{1}{2}\xi _3+\xi _{4}^{3}-\frac{1}{2}\xi _4 \right|\\
			=&\left| -2\left( \xi _{1}^{3}+\xi _{2}^{3} \right) -3\left( \xi _1+\xi _2 \right) \left( \xi _1+\xi _3 \right) \left( \xi _2+\xi _3 \right) +\xi _1+\xi _2 \right|\\
			=&2\left( \left| \xi _1 \right|^3+\left| \xi _2 \right|^3 \right) +3\left( \left| \xi _1 \right|+\left| \xi _2 \right| \right) \left( \left| \xi _3 \right|-\left| \xi _1 \right| \right) \left( \left| \xi _3 \right|-\left| \xi _2 \right| \right) -\left( \left| \xi _1 \right|+\left| \xi _2 \right| \right)\gtrsim\left| \xi _2 \right|^3\gtrsim\left| \xi _1\xi _2\xi _3 \right|.
			\end{aligned}
			\end{equation*}
			\item \textbf{Case A.(c)} : $\tau _1\ge 0,\tau _2\ge 0,\tau _3\ge 0,\tau _4\le 0,\xi _1\ge 0,\xi _2\ge 0,\xi _3\le 0,\xi _4\le 0 $.
			\begin{equation*}
			\begin{aligned}
			\left| H \right|=&\left| -\xi _{1}^{3}+\frac{1}{2}\xi _1-\xi _{2}^{3}+\frac{1}{2}\xi _2+\xi _{3}^{3}-\frac{1}{2}\xi _3-\xi _{4}^{3}+\frac{1}{2}\xi _4 \right|\\
			=&\left| 2\xi _{3}^{3}+3\left( \xi _1+\xi _2 \right) \left( \xi _1+\xi _3 \right) \left( \xi _2+\xi _3 \right) -\xi _3 \right|,\\
			=& 2\left| \xi _3 \right|^3-3\left( \left| \xi _1 \right|+\left| \xi _2 \right| \right) \left( \left| \xi _2 \right|-\left| \xi _4 \right| \right) \left( \left| \xi _1 \right|-\left| \xi _4 \right| \right) -\left| \xi _3 \right| \\
			=& 2\left( \left| \xi _1 \right|^3+\left| \xi _2 \right|^3-\left| \xi _4 \right|^3 \right) +3\left( \left| \xi _1 \right|+\left| \xi _2 \right| \right) \left( \left| \xi _2 \right|-\left| \xi _4 \right| \right) \left( \left| \xi _1 \right|-\left| \xi _4 \right| \right) -\left| \xi _3 \right| \\
			\gtrsim&\left| \xi _1 \right|\left| \xi _2 \right|^2\sim\left| \xi _1\xi _2\xi _3 \right|.\\
			\end{aligned}
			\end{equation*}
			\item \textbf{Case B.(b)} :  $\tau _1\ge 0,\tau _2\ge 0,\tau _3\le 0,\tau _4\ge 0,\xi _1\ge 0,\xi _2\ge 0,\xi _3\le 0,\xi _4\le 0 $.
			\begin{equation*}
			\begin{aligned}
			\left| H \right|=&-\xi _{1}^{3}+\frac{1}{2}\xi _1-\xi _{2}^{3}+\frac{1}{2}\xi _2-\xi _{3}^{3}+\frac{1}{2}\xi _3-\xi _{4}^{3}+\frac{1}{2}\xi _4\\
			=&3\left( \xi _1+\xi _2 \right) \left( \xi _1+\xi _3 \right) \left( \xi _2+\xi _3 \right)\\
			=&3\left( \left| \xi _1 \right|+\left| \xi _2 \right| \right) \left( \left| \xi _2 \right|+\left| \xi _4 \right| \right) \left( \left| \xi _1 \right|+\left| \xi _4 \right| \right)\gtrsim\left| \xi _1 \right|\left| \xi _2 \right|^2\sim\left| \xi _1\xi _2\xi _3 \right|.
			\end{aligned}
			\end{equation*}
			\item \textbf{Case B.(c)}:  $\tau _1\ge 0,\tau _2\ge 0,\tau _3\le 0,\tau _4\ge 0,\xi _1\ge 0,\xi _2\ge 0,\xi _3\le 0,\xi _4\le 0.$\\ 
			\begin{equation*}
			\begin{aligned}
			\left| H \right|=&\left| -\xi _{1}^{3}+\frac{1}{2}\xi _1-\xi _{2}^{3}+\frac{1}{2}\xi _2-\xi _{3}^{3}+\frac{1}{2}\xi _3+\xi _{4}^{3}-\frac{1}{2}\xi _4 \right|\\
			=&\left| -2\left( \xi _{1}^{3}+\xi _{2}^{3}+\xi _{3}^{2} \right) -3\left( \xi _1+\xi _2 \right) \left( \xi _1+\xi _3 \right) \left( \xi _2+\xi _3 \right) -\xi _4 \right|\\
			=&2\left( \left| \xi _3 \right|^3-\left| \xi _1 \right|^3-\left| \xi _2 \right|^3 \right) -3\left( \left| \xi _1 \right|+\left| \xi _2 \right| \right) \left( \left| \xi _3 \right|-\left| \xi _1 \right| \right) \left( \left| \xi _3 \right|-\left| \xi _2 \right| \right) +\left| \xi _4 \right|\\
			=&-2\left| \xi _4 \right|^3+3\left( \left| \xi _1 \right|+\left| \xi _2 \right| \right) \left( \left| \xi _2 \right|-\left| \xi _4 \right| \right) \left( \left| \xi _1 \right|-\left| \xi _4 \right| \right) +\left| \xi _4 \right| \gtrsim\left| \xi _1 \right|\left| \xi _2 \right|^2\sim\left| \xi _1\xi _2\xi _3 \right|.
			\end{aligned}
			\end{equation*}
		\end{itemize}
		\textbf{Case C} to \textbf{Case G} can be checked similarly, we omit for simplicity. 
	\end{proof}
	
	\section*{Acknowledgements}
	S.  Li is supported by the National Natural Science Foundation of China (No.  12001084 and No.  12071061).

\end{document}